\documentclass[11pt,twoside]{amsart}
\usepackage{latexsym,amssymb,amsmath}

\numberwithin{equation}{section}
\newtheorem{theorem}{Theorem}[section]
\newtheorem{lemma}[theorem]{Lemma}
\newtheorem{proposition}[theorem]{Proposition}
\newtheorem{corollary}[theorem]{Corollary}

\theoremstyle{definition}
\newtheorem{definition}[theorem]{Definition}

\newtheorem{example}[theorem]{Example}

\begin{document}

%%%%%%%%%%%%%%%%%%%%%%%%%%%%%%%%%%%%%%%%%%%%%%%%%%%%%%%%%%%%%%%%%%%%%

\title[]{Regularity and algebraic properties of certain \\ lattice ideals} 

\thanks{The first author was partially funded by CMUC and FCT
(Portugal), through 
European program COMPETE/FEDER and project PTDC/MAT/111332/2009, and by a
research grant from Santander Totta Bank (Portugal). The second
author is a member of the Center for Mathematical Analysis, Geometry,
and Dynamical Systems,  
Departamento de Matematica, Instituto Superior Tecnico, 1049-001 Lisboa, 
Portugal. The third author was partially supported by SNI}

\author{Jorge Neves}
\address{CMUC, Department of Mathematics, University of Coimbra
3001-454 Coimbra, Portugal.
}
\email{neves@mat.uc.pt}

\author{Maria Vaz Pinto}
\address{
Departamento de Matem\'atica\\
Instituto Superior T\'ecnico\\
Universidade T\'ecnica de Lisboa\\
Avenida Rovisco Pais, 1\\
1049-001 Lisboa, Portugal.
}
\email{vazpinto@math.ist.utl.pt}

\author{Rafael H. Villarreal}
\address{
Departamento de
Matem\'aticas\\
Centro de Investigaci\'on y de Estudios
Avanzados del
IPN\\
Apartado Postal
14--740 \\
07000 Mexico City, D.F.
}
\email{vila@math.cinvestav.mx}

%\urladdr{http://www.math.cinvestav.mx/$\sim$vila/}

\keywords{Lattice ideals, regularity, degree, complete intersections,
vanishing ideals}
\subjclass[2010]{Primary 13F20; Secondary 13P25, 14H45, 11T71.} 

\dedicatory{Dedicated to Professor Aron Simis on the occasion of his
$70$th birthday}

\begin{abstract} We study the regularity and the algebraic properties
of certain lattice ideals. 
We establish a map $I\mapsto \widetilde{I}$ between
the family of graded lattice ideals in an $\mathbb{N}$-graded
polynomial ring over a field $K$ and the family of graded lattice
ideals in a polynomial ring with the 
standard grading. This map is shown to preserve the complete
intersection property and the regularity of $I$ but not the degree.
We relate 
the Hilbert series and the generators of $I$ and
$\widetilde{I}$. If $\dim(I)=1$, we relate the degrees of
$I$ and $\widetilde{I}$. It is shown that the regularity of certain
lattice ideals is additive in a certain sense. Then, we give some
applications. For finite fields, we give a formula for the regularity
of  the vanishing ideal of a degenerate torus in terms of the
Frobenius number of a semigroup. We construct vanishing ideals, over
finite fields, with prescribed regularity and degree of a certain
type. Let $X$ be a subset of a projective space over a field $K$. It
is shown that the vanishing 
ideal of $X$ is a lattice ideal of dimension $1$ if and only if $X$ is
a finite subgroup of a projective torus. For finite fields, it is
shown that $X$ is a  subgroup of a projective torus if and only if $X$
is parameterized by monomials. We express the regularity of
the vanishing ideal over a bipartite graph in terms of the  
regularities of 
the vanishing ideals of the blocks of the graph.
\end{abstract}

\maketitle 

\section{Introduction}\label{intro-ci-mcurves}

Let $S=K[t_1,\ldots,t_s]=\oplus_{d=0}^\infty
S_d$ and $\widetilde{S}=K[t_1,\ldots,t_s]=\oplus_{d=0}^\infty 
\widetilde{S}_d$ be polynomial rings, over a field $K$, with the
gradings induced by setting $\deg(t_i)=d_i$ for
all $i$ and $\deg(t_i)=1$ for all $i$, respectively, where
$d_1,\ldots,d_s$ are positive integers. 
Let $F=\{f_1,\ldots,f_s\}$ be a set
of algebraically independent 
homogeneous polynomials of 
$\widetilde{S}$ of degrees $d_1,\ldots,d_s$ and let 
$$\phi\colon S\rightarrow K[F]$$ 
be the isomorphism of 
$K$-algebras given by $\phi(g)=g(f_1,\ldots,f_s)$, where
$K[F]$ is the $K$-subalgebra of $\widetilde{S}$ generated by
$f_1,\ldots,f_s$. For convenience we
denote $\phi(g)$ by $\widetilde{g}$. Given a graded ideal 
$I\subset S$ generated by $g_1,\ldots,g_m$, we associate to $I$ the
graded ideal $\widetilde{I}\subset \widetilde{S}$ generated by
$\widetilde{g}_1,\ldots,\widetilde{g}_m$ and call
$\widetilde{I}$ the {\it homogenization} of $I$ with respect to 
$f_1,\ldots,f_s$. The ideal $\widetilde{I}$ is independent of the generating 
set $g_1,\ldots,g_m$ and $\widetilde{I}$ is a graded ideal with
respect to the standard grading of $K[t_1,\ldots,t_s]$. 

If $f_i=t_i^{d_i}$ for $i=1,\ldots,s$, the map
$I\mapsto\widetilde{I}$ induces a correspondence between the family of
graded lattice ideals of $S$ and the family of graded lattice ideals of
$\widetilde{S}$. The first aim of this paper is to study this correspondence
and to relate the algebraic invariants (regularity and degree) and
properties of $I$ and $\widetilde{I}$ (especially the complete
intersection property). For finite fields, the interest in this
correspondence comes  
from the fact that any vanishing ideal
$I(X)\subset\widetilde{S}$, 
over a degenerate projective
torus $X$, arises as a toric ideal $I\subset S$ of a monomial curve, i.e.,
$\widetilde{I}=I(X)$ for some graded toric ideal $I$ of $S$  of dimension $1$ 
(see \cite[Proposition~3.2]{ci-mcurves}). In this paper, we extend
the scope of \cite{ci-mcurves} to include lattice ideals of arbitrary
dimension. The second aim of this paper is to use our methods to 
study the regularity of graded vanishing ideals and to give classifications of this
type of ideals. The algebraic invariants (degree, regularity) and the
complete intersection property of vanishing ideals over finite fields, are
of interest in algebraic coding theory \cite{GRT, algcodes,sorensen} and commutative
algebra
\cite{stcib-algorithm,bermejo-gimenez,Eisen,FMS,geramita-cayley-bacharach,
morales-thoma,Vas1}. The length, dimension and minimum distance of evaluation codes
arising from complete 
intersections have 
been studied in \cite{duursma-renteria-tapia,gold-little-schenck,
hansen,cartesian-codes,ci-codes}.

The contents of this paper are as follows. In Section~\ref{prelim}, we 
introduce the notions of degree and index of regularity via Hilbert
functions. Lattices and their lattice ideals are also introduced in this
section.  We present some of the results
that will be needed throughout the paper. All the results of this 
section are well known. 

In Section~\ref{hsai}, we establish a
map $I\mapsto \widetilde{I}$ between the graded ideals of
$S$ and $\widetilde{S}$. We relate the minimal graded resolutions, the
Hilbert series, and the regularities of $I$ and $\widetilde{I}$. 
In general, the map $I\mapsto \widetilde{I}$ does not preserve the
height of $I$ (Example~\ref{dec30-12-4}). Let $I$ be a graded ideal of $S$ and
assume that 
$K[F]\subset\widetilde{S}$ is an integral extension. We show that 
$\dim(S/I)=\dim(\widetilde{S}/\widetilde{I})$ (Lemma~\ref{jul8-12}). 
Then, using the Buchsbaum-Eisenbud acyclicity criterion
\cite[Theorem~1.4.13]{BHer}, we show that if  
\[
\ \ \ \ \ \ \ \ \ \ \ \ \textstyle 0\rightarrow 
\bigoplus_{j=1}^{b_g} S(-a_{gj})
{\rightarrow}\cdots
\rightarrow\bigoplus_{j=1}^{b_1}
S(-a_{1j}){\rightarrow} S\rightarrow
S/I\rightarrow 0 
\]
is the minimal graded free resolution of $S/I$, then 
\[
\ \ \ \ \ \ \ \ \ \ \ \ 
\ \ \ \ \ \ \ \ \ \ \ \ \textstyle 0\rightarrow 
\bigoplus_{j=1}^{b_g} \widetilde{S}(-a_{gj})
{\rightarrow}\cdots
\rightarrow\bigoplus_{j=1}^{b_1}
\widetilde{S}(-a_{1j}){\rightarrow}
\widetilde{S}\rightarrow \widetilde{S}/\widetilde{I} \rightarrow 0
\]
is the minimal graded free resolution of
$\widetilde{S}/\widetilde{I}$ (Lemma~\ref{jul9-12}). 
By the {\it regularity} of $S/I$, denoted by ${\rm reg}(S/I)$, 
we mean the Castelnuovo-Mumford regularity. This notion is introduced in
Section~\ref{hsai}. We denote the Hilbert series of $S/I$  by
$F_I(t)$.

This close relationship between the graded resolutions of $I$ and
$\widetilde{I}$ allows us to relate the Hilbert series and the regularities of
$I$ and $\widetilde{I}$.
 
\noindent {\bf Theorem~\ref{aug23-12}}{\it\ Let $I$ be a graded ideal
of $S$, then ${\rm reg}(S/I)={\rm reg}(\widetilde{S}/\widetilde{I})$
and 

\ \ \ \ \ \ \ $F_{\widetilde{I}}(t)=\lambda_1(t)\cdots\lambda_s(t)F_I(t)$, where
$\lambda_i(t)=1+t+\cdots+t^{d_i-1}$. 
}

For the rest of the introduction we will assume that $f_i=t_i^{d_i}$ for
$i=1,\ldots,s$. Accordingly, $\widetilde{I}\subset\widetilde{S}$ will
denote the homogenization 
of a graded ideal $I\subset S$ with respect to
$t_1^{d_1},\ldots,t_s^{d_s}$. 

In Section~\ref{ci-section}, we examine the map $I\mapsto
\widetilde{I}$ between the family of graded
lattice ideals of $S$ and the family of graded lattice ideals of 
$\widetilde{S}$. Let $\mathcal{L}\subset\mathbb{Z}^s$ be a {\it 
homogeneous lattice\/}, with respect to $d_1,\ldots,d_s$, and let
$I(\mathcal{L})\subset S$ be its graded lattice ideal. 
It is well known that the height of $I(\mathcal{L})$ is the rank of
$\mathcal{L}$ (see for instance \cite[Proposition~7.5]{cca}). 
If $D$ is the
non-singular diagonal matrix ${\rm diag}(d_1,\ldots,d_s)$, then
$\widetilde{I(\mathcal{L})}$ is the lattice ideal of 
$\widetilde{\mathcal{L}}=D(\mathcal{L})$ and the height of
$I(\mathcal{L})$ is equal to 
the height of $I(\widetilde{\mathcal{L}})$. 

We come to the main result of Section~\ref{ci-section} that relates the 
generators of $I(\mathcal{L})$ and $\widetilde{I(\mathcal{L})}$. 

\noindent {\bf Theorem~\ref{dec22-11}}{\it\ Let
$\mathcal{B}=\{t^{a_i}-t^{b_i}\}_{i=1}^m$ be a set of binomials. If 
$\widetilde{\mathcal{B}}=\{t^{D(a_i)}-t^{D(b_i)}\}_{i=1}^m$, then
$I(\mathcal{L})=(\mathcal{B})$ if and only if
$I(\widetilde{\mathcal{L}})=(\widetilde{\mathcal{B}})$. 
}

Then, using the fact that ${\rm ht}(I(\mathcal{L}))={\rm
ht}(I(\widetilde{\mathcal{L}}))$, we show that $I(\widetilde{\mathcal{L}})$ is
a complete intersection if 
and only if $I(\mathcal{L})$ is a complete intersection
(Corollary~\ref{ci-i(x)-p}). 

In Section~\ref{lattice-dim1-section}, we study the restriction of the
map $I\mapsto\widetilde{I}$ to the family of graded lattice ideals of
dimension $1$. In this case, the degrees of $I$ and $\widetilde{I}$ are nicely
related.

\noindent {\bf Theorem~\ref{dec5-12}}{\it\ If $I$
is a graded lattice ideal of dimension $1$, 
then 
$$ 
\deg(\widetilde{S}/\widetilde{I})=\frac{d_1\cdots
d_s}{\max\{d_1,\ldots,d_s\}}\deg(S/I).
$$
}

Let $r=\gcd(d_1,\ldots,d_s)$ and let $\mathcal{S}$ be the numerical semigroup
$\mathbb{N}(d_1/r)+\cdots+\mathbb{N}(d_s/r)$. The {\it Frobenius number\/} 
of $\mathcal{S}$, denoted by $g(\mathcal{S})$, is the largest 
integer not in $\mathcal{S}$. The Frobenius
number occurs in many branches of mathematics and is one of the most
studied invariants in the theory of 
semigroups. A great deal of effort has been directed at the 
effective computation of this number, see the monograph of Ram\'\i
rez-Alfons\'\i n \cite{frobenius-problem}.

The next result gives an explicit formula for the
regularity of $\widetilde{I}$, in terms of the Frobenius number of
$\mathcal{S}$, when $I$ is the
toric ideal of a monomial curve. This formula can be used to
compute the regularity using some available algorithms 
to compute Frobenius numbers \cite{frobenius-problem}.

\noindent {\bf Theorem~\ref{reg-deg-mcurves}}{\it\  If $I$ is the toric ideal of
$K[y_1^{d_1},\ldots,y_1^{d_s}]\subset K[y_1]$, then 
$$
{\rm reg}(\widetilde{S}/\widetilde{I})=r\cdot
g(\mathcal{S})+1+\sum_{i=1}^s(d_i-1). 
$$
}
Let $\succ$ be the reverse lexicographical order. If $I$ is a graded lattice
ideal and $\dim(S/I)=1$, we show the following equalities
$$
{\rm reg}(S/I)={\rm reg}(\widetilde{S}/\widetilde{I})=
{\rm reg}(\widetilde{S}/{\rm in}(\widetilde{I}))
={\rm reg}(S/{\rm in}(I)),
$$
where ${\rm in}({I})$, ${\rm in}(\widetilde{I})$ are the initial
ideals of $I$, $\widetilde{I}$, with respect to $\succ$, respectively
(Corollary~\ref{nov30-12}). 

We come to the last main result of Section~\ref{lattice-dim1-section}
showing that the 
regularity 
of the saturation of
certain one dimensional graded ideals is additive in a certain sense.
This will be used in Section~\ref{applications-to-i(x)-1} to study the
regularity of vanishing ideals over bipartite graphs.

\noindent {\bf Theorem~\ref{vila-12-12-12}}{\it\  
Let $V_1,\ldots,V_c$ be a partition of 
$V=\{t_1,\ldots,t_s\}$ and let $\ell$ be a positive integer. If $I_k$
is a graded binomial ideal of $K[V_k]$ such that
$t_i^\ell-t_j^\ell\in I_k$ for 
$t_i,t_j\in V_k$ and $\mathcal{I}$ is the ideal of $K[V]$ generated by
all binomials $t_i^{\ell}-t_j^{\ell}$ with $1\leq i,j\leq s$, then 
\begin{equation*}
{\rm reg}\, K[V]/(I_1+\cdots+I_c+\mathcal{I}\colon
h^\infty)= 
  \sum_{k=1}^c{\rm reg}\, K[V_k]/(I_k\colon h_k^\infty)+(c-1)(\ell-1),
\end{equation*}
where $h=t_1\cdots t_s$ and $h_k=\prod_{t_i\in V_k}t_i$ for
$k=1,\ldots,c$. Here $K[V_k]$ and $K[V]$ are polynomial rings with 
the standard grading.}

In Section~\ref{applications-to-i(x)}, we study graded vanishing
ideals over arbitrary fields and give 
some applications of the results of
Section~\ref{lattice-dim1-section}. 
Let $K$  be a field and let $\mathbb{P}^{s-1}$ be
the projective space of dimension $s-1$ over $K$. Given a sequence
$v=(v_1,\ldots,v_s)$ of  
positive integers, the set 
$$
\{[(x_1^{v_{1}},\ldots,x_s^{v_s})]\, \vert\, x_i\in K^*\mbox{ for all
}i\}\subset\mathbb{P}^{s-1}
$$
is called a {\it degenerate projective torus\/} of type $v$, 
where $K^*=K\setminus\{0\}$. If $v_i=1$ for all $i$, this set 
is called a {\it projective torus\/} in $\mathbb{P}^{s-1}$ and it is
denoted by $\mathbb{T}$. If $X$ is a subset of $\mathbb{P}^{s-1}$, 
the {\it vanishing ideal\/} of $X$, denoted by
$I(X)$, is the ideal of $\widetilde{S}$ generated by the homogeneous
polynomials that vanish on all $X$.

For finite fields, we give formulae for the degree and regularity of
graded vanishing ideals over degenerate tori. 
The next result was shown in \cite{ci-mcurves} under the hypothesis that 
$I(X)$ is a complete intersection.  

\noindent {\bf Corollary~\ref{dec9-12}}{\it\ Let $K=\mathbb{F}_q$ be
a finite field and let $X$ 
be a degenerate projective torus of 
type $v=({v_1},\ldots,v_s)$. If $d_i=(q-1)/\gcd(v_i,q-1)$ for
$i=1,\ldots,s$ and $r=\gcd(d_1,\ldots,d_s)$, then 
$$
{\rm
reg}(\widetilde{S}/I(X))=r\cdot g(\mathcal{S})+1+\sum_{i=1}^s(d_i-1)\
\mbox{ and }\ \deg(\widetilde{S}/I(X))=d_1\cdots d_s/r.
$$
} 
This result allows us to
construct graded vanishing ideals over finite fields with prescribed 
regularity and degree of a
certain type (Proposition~\ref{dec30-12-3}). 

We characterize when a graded lattice ideal of dimension $1$ is a
vanishing ideal in terms of the degree (Proposition~\ref{dec30-12-2}).
Then we classify the vanishing ideals that are lattice ideals of dimension $1$.
For finite fields, it is shown that $X$ is a  subgroup of a 
projective torus if and only if $X$
is parameterized by monomials (Proposition~\ref{dec30-12-1}). For
infinite, fields we show a formula for the vanishing ideal of 
an algebraic toric set parameterized by monomials
(Theorem~\ref{ipn-ufpe-cinvestav-1}). For finite fields, a formula for
the vanishing ideal was shown in \cite[Theorem~2.1]{algcodes}.

In Section~\ref{applications-to-i(x)-1}, 
we study graded vanishing ideals over bipartite
graphs. Let $G$ be a simple graph with vertex set $V_G=\{y_1,\ldots,y_n\}$
and edge set $E_G$.  We refer to \cite{Boll} for the general theory of
graphs. Let $\{v_1,\ldots,v_s\}\subset \mathbb{N}^n$ be the 
set of all characteristic vectors of the edges of the graph $G$. We
may identify the edges of $G$ with the variables $t_1,\ldots,t_s$ of a
polynomial ring $K[t_1,\ldots,t_s]$. The set $X\subset\mathbb{T}$
parameterized by 
the monomials $y^{v_1},\ldots,y^{v_s}$ is called the {\it projective
algebraic 
toric set parameterized by the edges} 
of $G$. 

For bipartite graphs, if the field $K$ is finite, 
we are able to express the regularity of
the vanishing ideal of $X$ in terms of the regularities of 
the vanishing ideals of the projective algebraic
toric sets parameterized by the edges of the blocks of the graph. The
blocks of $G$ are essentially the maximal $2$-connected subgraphs of $G$ (see
Section~\ref{applications-to-i(x)-1}). If $K$ is infinite, the
vanishing ideal of $X$ is the toric ideal of
$K[t^{v_1}z,\ldots,t^{v_s}z]$ (as is seen in
Theorem~\ref{ipn-ufpe-cinvestav-1}), in this case $I(X)$ is 
Cohen-Macaulay \cite{ITG,bowtie} (because $G$ is bipartite) and formulae 
for the regularity and the $a$-invariant are given in
\cite[Proposition~4.2]{shiftcon}. 

We come to the main result of Section~\ref{applications-to-i(x)-1}.

\noindent {\bf Theorem~\ref{maria-jorge-vila-12-12-12}}{\it\   
Let $G$ be a bipartite graph without isolated vertices
and let $G_1,\ldots,G_c$ be the blocks of $G$. If $K$ is a finite
field with $q$ elements and $X$ $($resp $X_k)$
is the projective algebraic toric set parameterized by
the edges of $G$ $($resp. $G_k\mathrm{)}$, then
$$
{\rm reg}\, K[E_G]/I(X)=\sum_{k=1}^c{\rm reg}\,
K[E_{G_k}]/I(X_k)+(q-2)(c-1).
$$
}
Let $P$ be
the toric ideal of $\mathbb{F}_q[y^{v_1},\ldots,y^{v_s}]$ and let $I$ 
be the binomial ideal $I=P+\mathcal{I}$, where
$\mathcal{I}$ is the binomial ideal 
$$\mathcal{I}=(\{t_i^{q-1}-t_j^{q-1}\vert\, 
t_i,t_j\in E_{G}\}).$$ 
We relate the
regularity of $I(X)$ with the Hilbert function of $S/I$ and the
primary decompositions of $I$ (Proposition~\ref{dec30-12}).  As a byproduct one
obtains a method that can be used to compute the regularity
(Proposition~\ref{dec30-12}(d)). For an
arbitrary bipartite graph, Theorem~\ref{maria-jorge-vila-12-12-12} and
\cite[Theorem 2.18]{rs-codes} can be used to bound the regularity
of $I(X)$. 

For all unexplained
terminology and additional information,  we refer to
\cite{EisStu,cca} (for the theory of 
binomial and lattice ideals),
\cite{CLO,Eisen,eisenbud-syzygies,singular-book,Sta1,Vas1}
(for commutative algebra and the
theory of Hilbert functions), and 
\cite{Boll} (for the theory of graphs).

\section{Preliminaries}\label{prelim}

In this section, we 
introduce the notions of degree and index of regularity---via Hilbert
functions---and the notion of a lattice ideal. We present some of the results
that will be needed throughout the paper. 

Let $S=K[t_1,\ldots,t_s]$ be a polynomial ring over a field $K$ and
let $I$ be an ideal of $S$. The vector space of polynomials of $S$
(resp. $I$) of degree at most $d$ is denoted by $S_{\leq d}$ (resp.
$I_{\leq d}$). The functions 
$$
H_I^a(d)=\dim_K(S_{\leq d}/I_{\leq d})\ \ \mbox{ and }\ \
H_I(d)=H_I^a(d)-H_I^a(d-1)
$$
are called the {\it affine Hilbert function} and the {\it Hilbert
function} of $S/I$ respectively. We denote the Krull dimension of
$S/I$ by $\dim(S/I)$. If $k=\dim(S/I)$, according 
to \cite[Remark~5.3.16, p.~330]{singular-book}, there are unique
polynomials 
$$
\textstyle h^a_I(t)=\sum_{i=0}^{k}a_it^i\in 
\mathbb{Q}[t]\ \mbox{ and }\ h_I(t)=\sum_{i=0}^{k-1}c_it^i\in
\mathbb{Q}[t]$$ 
of degrees $k$ and $k-1$, respectively, such that
$h^a_I(d)=H_I^a(d)$ and $h_I(d)=H_I(d)$ for 
$d\gg 0$.  By convention the zero polynomial has degree $-1$. 

\begin{definition}
The integer $a_k(k!)$, denoted by ${\rm deg}(S/I)$, is 
called the {\it degree\/} of $S/I$. 
\end{definition}

Notice that $a_k(k!)=c_{k-1}((k-1)!)$ for $k\geq 1$. If $k=0$, then
$H_I^a(d)=\dim_K(S/I)$ for $d\gg 0$ and the degree of $S/I$ is 
just $\dim_K(S/I)$. If $S=\oplus_{d=0}^{\infty}S_d$ has the standard
grading and 
$I$ is a graded ideal, then
$$
\textstyle H_I^a(d)=\sum_{i=0}^d\dim_K(S_d/I_d)
$$
where $I_d=I\cap S_d$. Thus, one has $H_I(d)=\dim_K(S_d/I_d)$ for
all $d$.

\begin{definition}\label{definition:index-of-regularity}
The \emph{index of regularity} of $S/I$, denoted by
${\rm r}(S/I)$, is the least integer $\ell\geq 0$ such that
$h_I(d)=H_I(d)$ for $d\geq \ell$. 
\end{definition}

If $S$ has the standard grading and $I$ is a graded
Cohen-Macaulay ideal of $S$ of dimension $1$, then ${\rm reg}(S/I)$,
the Castelnuovo 
Mumford regularity of $S/I$, is equal to the index of regularity of $S/I$ (see
Theorem~\ref{jul10-12}).  

\begin{proposition}{\rm(\cite[Lemma~5.3.11]{singular-book}, 
\cite{prim-dec-critical})}\label{additivity-of-the-degree}
If $I$ is an ideal of $S$ and 
$I=\mathfrak{q}_1\cap\cdots\cap\mathfrak{q}_m$ is a minimal primary
decomposition, then
$$
\deg(S/I)=\sum_{{\rm ht}(\mathfrak{q}_i)={\rm
ht}(I)}\deg(S/\mathfrak{q}_i).$$
\end{proposition}

\begin{definition} Let $\mathbb{P}^{s-1}$ be a projective space over
$K$ and let $X\subset\mathbb{P}^{s-1}$. If $S$ has the standard
grading, the {\it vanishing ideal\/} of $X$, denoted by
$I(X)$, is the ideal of $S$ generated by the homogeneous
polynomials of $S$ that vanish on all $X$.  
\end{definition}

\begin{corollary}{\cite{geramita-cayley-bacharach}}  
If $X\subset\mathbb{P}^{s-1}$ is a finite set, 
then $\deg(S/I(X))=|X|$.

\end{corollary} 

Recall that a binomial in $S$ is a polynomial of the
form $t^a-t^b$, where 
$a,b\in \mathbb{N}^s$ and where, if
\mbox{$a=(a_1,\dots,a_s)\in\mathbb{N}^s$}, we set 
\[
t^a=t_1^{a_1}\cdots t_s^{a_s}\in S. 
\]
A binomial of the form $t^a-t^b$ is usually referred 
to as a {\it pure binomial\/} \cite{EisStu}, although here we are dropping the
adjective ``pure''.  A {\it binomial ideal\/} is an ideal generated by binomials. 

Given $c=(c_i)\in {\mathbb Z}^s$, the set ${\rm supp}(c)=\{i\, |\,
c_i\neq 0\}$ is  called the {\it support\/} of
$c$. The vector $c$ can be uniquely written as $c=c^+-c^-$, 
where $c^+$ and $c^-$ are two nonnegative vectors 
with disjoint support, the {\it positive\/} and 
the {\it negative\/} part of $c$ respectively. If $t^a$ is a monomial,
with $a=(a_i)\in\mathbb{N}^s$, the set 
${\rm supp}(t^a)=\{t_i\vert\, a_i>0\}$ is called the {\it support\/} of $t^a$.

\begin{definition}\label{lattice-ideal-def}\rm 
A {\it lattice ideal\/} is an ideal of the form 
$I(\mathcal{L})=(t^{a^+}-t^{a^-}\vert\, 
a\in\mathcal{L})\subset S$ for some subgroup 
$\mathcal{L}$ of $\mathbb{Z}^s$. A subgroup $\mathcal{L}$ of $\mathbb{Z}^s$ is
called a {\it lattice\/}.   
\end{definition}

The class of lattice ideals has been studied in many places, see for
instance \cite{EisStu,cca} and the references 
there. This concept is a natural generalization of a toric ideal. 

The following is a well known description of lattice
ideals that follows from \cite[Corollary~2.5]{EisStu}.

\begin{theorem}{\rm\cite{EisStu}}\label{jun12-02} If $I$ is a binomial ideal 
of $S$, then $I$ is a lattice ideal if and only if $t_i$ is a non-zero
divisor of $S/I$ 
for all $i$. 
\end{theorem}

Given a subset $I\subset S$, its {\it variety\/}, denoted by $V(I)$,
is the set of all  
$a\in\mathbb{A}_K^a$ such that $f(a)=0$ for all $f\in I$, where
$\mathbb{A}_K^a$ is the affine space over $K$. 

\begin{lemma}{\rm\cite{ci-lattice}}\label{apr24-12-1} Let $I\subset S$ be a graded
binomial ideal such that $V(I,t_i)=\{0\}$ for all $i$. Then the
following hold.

{\rm (a)} If $I$ is Cohen-Macaulay, then $I$ is a lattice ideal.

{\rm (b)} If $\mathfrak{p}$ is a prime ideal
containing $(I,t_k)$ for some $1\leq k\leq s$, then $\mathfrak{p}=(t_1,\ldots,t_s)$.
\end{lemma}

\section{Hilbert series and algebraic invariants}\label{hsai}

We continue to use the notation and definitions used in
Section~\ref{intro-ci-mcurves}. In this section we establish a
map $I\mapsto \widetilde{I}$ between the graded ideals of
$S$ and $\widetilde{S}$. We relate the minimal graded resolutions, the
Hilbert series, and the regularities of $I$ and $\widetilde{I}$.

In what follows $S=K[t_1,\ldots,t_s]=\oplus_{d=0}^\infty
S_d$ and $\widetilde{S}=K[t_1,\ldots,t_s]=\oplus_{d=0}^\infty
\widetilde{S}_d$ are polynomial rings 
graded by the grading induced by setting $\deg(t_i)=d_i$ for
all $i$ and the standard grading induced by
setting $\deg(t_i)=1$ for all $i$, respectively, where
$d_1,\ldots,d_s$ are positive integers. Let $F=\{f_1,\ldots,f_s\}$ be a set
of algebraically independent 
homogeneous polynomials of 
$\widetilde{S}$ of degrees $d_1,\ldots,d_s$ and let 
$$\phi\colon S\rightarrow K[F]$$ 
be the isomorphism of 
$K$-algebras given by $\phi(g)=g(f_1,\ldots,f_s)$, where
$K[F]$ is the $K$-subalgebra of $\widetilde{S}$ generated by
$F$. For convenience we
denote $\phi(g)$ by $\widetilde{g}$. 

\begin{definition}
Given a graded ideal 
$I\subset S$ generated by $g_1,\ldots,g_m$, we associate to $I$ the
graded ideal $\widetilde{I}\subset \widetilde{S}$ generated by
$\widetilde{g}_1,\ldots,\widetilde{g}_m$. We call
$\widetilde{I}$ the {\it homogenization} of $I$ with respect to 
$f_1,\ldots,f_s$. 
\end{definition}

The ideal $\widetilde{I}$ is independent of the generating 
set $g_1,\ldots,g_m$ and $\widetilde{I}$ is a graded ideal with
respect to the standard grading of $K[t_1,\ldots,t_s]$. In general the
map $I\mapsto \widetilde{I}$ does not preserve the height of $I$, as
the following example shows. 

\begin{example}\label{dec30-12-4} The polynomials $f_1=t_1$, $f_2=t_2$,
$f_3=t_1t_2-t_1t_3$ are algebraically independent over $\mathbb{Q}$.
The homogenization of $I=(t_1,t_2,t_3)$ with respect to 
$f_1,f_2,f_3$ is $\widetilde{I}=(t_1, t_2,t_1t_2-t_1t_3)$ which is
equal to $(t_1,t_2)$. Thus, ${\rm ht}(\widetilde{I})<{\rm ht}(I)$. 
\end{example}

\begin{lemma}\label{jul8-12} If
$K[F]\subset\widetilde{S}$ is 
an integral extension and $I$ is a graded ideal of $S$, then 
$\dim(S/I)=\dim(\widetilde{S}/\widetilde{I})$. 
\end{lemma}

\begin{proof} Notice that $\phi\colon S\rightarrow K[F]$ is an isomorphism
and therefore, $\dim(S/I) = \dim(K[F]/\phi(I))$. But
$\widetilde{I}=\phi(I)\widetilde{S}$ and $K[F]$ is normal and
$\widetilde{S}$ is a
domain, then the integral extension $K[F]\subset\widetilde{S}$ satisfies
the going-down and the lying-over conditions (see \cite[Theorem~5]{Mat}).
Hence, by \cite[Proposition~B.2.4]{huneke-swanson-book}, 
${\rm ht}(\phi(I))={\rm ht}(\phi(I)\widetilde{S})$. Consequently, 
$\dim(K[F]/\phi(I))=\dim(\widetilde{S}/\widetilde{I})$.
\end{proof}

Let $I$ be a graded ideal of $S$ and let $F_I(t)$ be the Hilbert
series of $S/I$. The $a$-{\it invariant\/} of
$S/I$, denoted by $a(S/I)$, is the degree of $F_I(t)$ as a rational
function. Let 
\[
\leqno{{\mathbf F}:}\ \ \ \ \ \ \ \ \ \ \ \ \textstyle 0\rightarrow 
\bigoplus_{j=1}^{b_g} S(-a_{g,j})
{\rightarrow}\cdots
\rightarrow\bigoplus_{j=1}^{b_1}
S(-a_{1,j}){\rightarrow} S\rightarrow
S/I\rightarrow 0 
\]
be the minimal graded free resolution of $S/I$ as an $S$-module. 
The free modules in the
resolution of $S/I$ can be written as
$$
\textstyle F_i=\bigoplus_{j=1}^{b_i}
{S}(-a_{i,j})=\bigoplus_{j}S(-j)^{b_{i,j}}.
$$
The numbers $b_{i,j}={\rm Tor}_i(K,S/I)_j$ are called the {\it graded
Betti numbers\/} of $S/I$ and $b_i=\sum_{j}b_{i,j}$ is called the
$i^\textup{th}$ {\it Betti number\/} of $S/I$. The 
{\it Castelnuovo-Mumford regularity\/} or simply the {\it regularity}
of $S/I$ is defined as 
$${\rm reg}(S/I)=\max\{j-i\vert\,
b_{i,j}\neq 0\}.
$$ 

\begin{theorem}{\rm(\cite[p.~521]{Eisen},
\cite[Proposition~4.2.3]{monalg})}\label{jul10-12} 
If $I\subset S$ is a graded Cohen-Macaulay ideal, then  
$$
\textstyle a(S/I)={\rm reg}(S/I)-{\rm
depth}(S/I)-\sum_{i=1}^s(d_i-1)={\rm reg}(S/I)+{\rm
ht}(I)-\sum_{i=1}^sd_i.
$$
\end{theorem}

\begin{lemma}\label{jul9-12} Let $I$ be a graded ideal of
$S$. If $K[f_1,\ldots,f_s]\subset\widetilde{S}$ is an integral
extension and 
\[
\leqno{{\mathbf F}:}\ \ \ \ \ \ \ \ \ \ \ \ \textstyle 0\rightarrow 
\bigoplus_{j=1}^{b_g} S(-a_{g,j})
{\rightarrow}\cdots
\rightarrow\bigoplus_{j=1}^{b_1}
S(-a_{1,j}){\rightarrow} S\rightarrow
S/I\rightarrow 0 
\]
is the minimal graded free resolution of $S/I$, then 
\[
\leqno{\widetilde{\mathbf{F}}:}\ \ \ \ \ \ \ \ \ \ \ \ 
\ \ \ \ \ \ \ \ \ \ \ \ \textstyle 0\rightarrow 
\bigoplus_{j=1}^{b_g} \widetilde{S}(-a_{g,j})
{\rightarrow}\cdots
\rightarrow\bigoplus_{j=1}^{b_1}
\widetilde{S}(-a_{1,j}){\rightarrow}
\widetilde{S}\rightarrow \widetilde{S}/\widetilde{I} \rightarrow 0
\]
is the minimal graded free resolution of
$\widetilde{S}/\widetilde{I}$. 
\end{lemma}

\begin{proof} For $1\leq i\leq g$ consider the map
$\varphi_i\colon\bigoplus_{j=1}^{b_i}
S(-a_{i,j}){\rightarrow} \bigoplus_{j=1}^{b_{i-1}}
S(-a_{i-1,i})$ of the resolution of $S/I$. The entries of the matrix
$\varphi_i$ are homogeneous polynomials of $S$, and accordingly the
ideal $L_i=I_{r_i}(\varphi_i)$ generated by the $r_i$-minors of $\varphi_i$
is graded with respect to the grading of $S$, where $r_i$ is the rank
of $\varphi_i$ (that is, the largest size of a nonvanishing minor). 

Let $\widetilde{\varphi}_i$ be the matrix obtained from $\varphi_i$ by replacing 
each entry of $\varphi_i$ by its image under the map $S\mapsto
\widetilde{S}$,
$g\mapsto\widetilde{g}=g(f_1,\ldots,f_s)$. Since this map
is an injective homomorphism of $K$-algebras, the rank of
$\widetilde{\varphi}_i$ is equal to $r_i$,
$\widetilde{\varphi}_{i-1}\widetilde{\varphi}_i=0$, and
$I_{r_i}(\widetilde{\varphi_i})=\widetilde{L_i}$ for all $i$.
Therefore, one has a graded complex 
\[
\textstyle 0\rightarrow 
\bigoplus_{j=1}^{b_g} \widetilde{S}(-a_{g,j})
\stackrel{\widetilde{\varphi_g}}{\rightarrow} \cdots 
\rightarrow\bigoplus_{j=1}^{b_i}
\widetilde{S}(-a_{i,j})\stackrel{\widetilde{\varphi}_i}{\rightarrow}\cdots
\rightarrow\bigoplus_{j=1}^{b_1}
\widetilde{S}(-a_{1,j})\stackrel{\widetilde{\varphi}_1}{\rightarrow}
\widetilde{S}\rightarrow \widetilde{S}/\widetilde{I} \rightarrow 0.
\]

To show that this complex is exact, by the Buchsbaum-Eisenbud acyclicity criterion
\cite[Theorem~1.4.13, p.~24]{BHer}, it suffices to verify that 
${\rm ht}(I_{r_i}(\widetilde{\varphi}_i))$ is at least $i$ for $i\geq 1$. As
${\mathbf F}$ is an exact complex, using Lemma~\ref{jul8-12}, we get
$$
{\rm ht}(I_{r_i}(\widetilde{\varphi}_i))={\rm
ht}(\widetilde{I_{r_i}(\varphi_i)})\geq {\rm
ht}(I_{r_i}(\varphi_i))\geq i,
$$
as required.
\end{proof}

\begin{theorem}\label{aug23-12} Let $I$ be a graded ideal of $S$ and
let $F_I(t)$ be the Hilbert series of $S/I$. If
$K[f_1,\ldots,f_s]\subset\widetilde{S}$ is an integral extension , then

$(\mathrm{a})$
$F_{\widetilde{I}}(t)=\lambda_1(t)\cdots\lambda_s(t)F_I(t)$, where
$\lambda_i(t)=1+t+\cdots+t^{d_i-1}$.

$(\mathrm{b})$ ${\rm reg}(S/I)={\rm
reg}(\widetilde{S}/\widetilde{I})$.
\end{theorem}

\begin{proof} (a): The Hilbert series of $\widetilde{S}$ and $S$,
denoted by $F(\widetilde{S},t)$ and $F(S,t)$ respectively, 
are
related by $F(\widetilde{S},t)=\lambda_1(t)\cdots\lambda_s(t)F(S,t)$.
Hence, using Lemma~\ref{jul9-12} and the additivity of Hilbert
series, we get the required equality. 

(b): This follows at once from Lemma~\ref{jul9-12}
\end{proof}

\begin{lemma}\label{neves-lemma}  
Let $I\subset S$ be a graded ideal and let
$f(t)/\prod_{i=1}^s(1-t^{d_i})$ be the 
Hilbert series of $S/I$, where $f(t)\in\mathbb{Z}[t]$. If $J \subset S$
is the ideal generated 
by all $g(t_1^r,\ldots,t_s^r)$ with $g\in I$, then 
\begin{itemize}
\item[(a)] $f(t^r)/\prod_{i=1}^s(1-t^{d_i})$ is the
Hilbert series of $S/J$.

\item[(b)] If $d_i=1$ for all $i$ and $I$ is a Cohen-Macaulay
ideal such that ${\rm ht}(I)={\rm ht}(J)$, then 
$${\rm reg}(S/J)={\rm ht}(I)(r-1)+r\cdot{\rm reg}(S/I).
$$
\end{itemize}
\end{lemma}

\begin{proof} (a): Clearly $J$ is also a graded ideal of $S$. If 
\[
\ \ \ \ \ \ \ \ \ \ \ \ \textstyle 0\rightarrow 
\bigoplus_{j=1}^{b_g} S(-a_{g,j})
{\rightarrow}\cdots
\rightarrow\bigoplus_{j=1}^{b_1}
S(-a_{1,j}){\rightarrow} S\rightarrow
S/I\rightarrow 0 
\]
is the minimal graded free resolution of $S/I$, then it is seen that
\[
\ \ \ \ \ \ \ \ \ \ \ \ \ \ \ \ \ \ \ \ \ \ \ \ \textstyle 0\rightarrow 
\bigoplus_{j=1}^{b_g}S(-ra_{g,j})
{\rightarrow}\cdots
\rightarrow\bigoplus_{j=1}^{b_1}
S(-ra_{1,j}){\rightarrow}
S\rightarrow S/J\rightarrow 0
\]
is the minimal graded free resolution of $S/J$. This can be shown
using the method of proof of Lemma~\ref{jul9-12}. Hence, by the
additivity of the Hilbert series, the result follows. 

(b): This follows from part (a) and Theorem~\ref{jul10-12}. 
\end{proof}

\section{Complete intersections and algebraic invariants}\label{ci-section}

We continue to use the notation and definitions used in
Section~\ref{intro-ci-mcurves}. In this section, we establish a 
map $I\mapsto \widetilde{I}$ between the family of graded
lattice ideals of a positively graded
polynomial ring and the graded lattice ideals of a polynomial ring
with the standard grading. We relate the lattices and generators of $I$ and
$\widetilde{I}$, then we show that $I$ is a complete intersection
if and only if $\widetilde{I}$ is a complete intersection. 

Let $K$ be a field and let $S=K[t_1,\ldots,t_s]=\oplus_{d=0}^\infty
S_d$ and $\widetilde{S}=K[t_1,\ldots,t_s]=\oplus_{d=0}^\infty
\widetilde{S}_d$ be polynomial rings 
graded by the grading induced by setting $\deg(t_i)=d_i$ for
all $i$ and the standard grading induced by
setting $\deg(t_i)=1$ for all $i$, respectively, where
$d_1,\ldots,d_s$ are positive integers.

Throughout this section, $\mathcal{L}\subset\mathbb{Z}^s$ will denote a {\it 
homogeneous lattice}, with respect to the positive vector 
$\mathbf{d}=(d_1,\ldots,d_s)$, i.e., $\langle\mathbf{d},a\rangle=0$
for $a\in\mathcal{L}$, and $I(\mathcal{L})\subset S$ will denote the 
graded lattice ideal of $\mathcal{L}$. 

Let $D$ be the non-singular diagonal matrix $D={\rm diag}(d_1,\ldots,d_s)$.
Consider the homomorphism of $\mathbb{Z}$-modules:
\begin{eqnarray*} 
D\colon\mathbb{Z}^s\rightarrow\mathbb{Z}^s,& e_i\mapsto d_ie_i.&
\end{eqnarray*}

The lattice $\widetilde{\mathcal{L}}=D(\mathcal{L})$ 
is called the {\it homogenization} of $\mathcal{L}$ with respect to 
$d_1,\ldots,d_s$.  Notice that
$\widetilde{I(\mathcal{L})}=I(\widetilde{\mathcal{L}})$, i.e.,
$I(\widetilde{\mathcal{L}})$ is the homogenization of $I(\mathcal{L})$
with respect to $t_1^{d_1},\ldots,t_s^{d_s}$. The lattices that
define the lattice ideals $I(\mathcal{L})$ and
$I(\widetilde{\mathcal{L}})$ are homogeneous with respect to
the vectors $\mathbf{d}=(d_1,\ldots,d_s)$ and
$\mathbf{1}=(1,\ldots,1)$, respectively. 

\begin{lemma}\label{dec19-11} The map $t^a-t^b\mapsto t^{D(a)}-t^{D(b)}$ induces a
bijection between the binomials $t^a-t^b$ of $I(\mathcal{L})$ whose terms $t^a$,
$t^b$ have disjoint support and the binomials $t^{a'}-t^{b'}$ of
$I(\widetilde{\mathcal{L}})$ whose terms
$t^{a'}$, $t^{b'}$ have disjoint support. 
\end{lemma}

\begin{proof} If $f=t^a-t^b$ is a binomial of $I(\mathcal{L})$ whose terms have
disjoint support, then $a-b\in\mathcal{L}$. Consequently, the terms of
$\widetilde{f}=t^{D(a)}-t^{D(b)}$ have disjoint support because 
$${\rm supp}(t^a)={\rm supp}(t^{D(a)})\  \mbox{ and }\ {\rm supp}(t^b)={\rm
supp}(t^{D(b)}),
$$ 
and $\widetilde{f}$ is in
$I(\widetilde{\mathcal{L}})$ because
$D(a)-D(b)\in\widetilde{\mathcal{L}}$. Thus,
the map is well defined.  The map is clearly
injective. To show that the map is onto, take a binomial
$f'=t^{a'}-t^{b'}$ in $I(\widetilde{\mathcal{L}})$ such
that $t^{a'}$ and $t^{b'}$ have disjoint support. Hence,
$a'-b'\in\widetilde{\mathcal{L}}$ and there is 
$c=c^+-c^-\in\mathcal{L}$ such that $a'-b'=D(c^+)-D(c^-)$. As $a'$
and $b'$ have disjoint support, we get $a'=D(c^+)$ and 
$b'=D(c^-)$. Thus, the binomial $f=t^{c^+}-t^{c^{-}}$ is in
$I(\mathcal{L})$ and maps to $t^{a'}-t^{b'}$.
\end{proof}

\begin{theorem}\label{dec22-11} Let
$\mathcal{B}=\{t^{a_i}-t^{b_i}\}_{i=1}^m$ be a set of binomials. If 
$\widetilde{\mathcal{B}}=\{t^{D(a_i)}-t^{D(b_i)}\}_{i=1}^m$, then
$I(\mathcal{L})=(\mathcal{B})$ if and only if
$I(\widetilde{\mathcal{L}})=(\widetilde{\mathcal{B}})$. 
\end{theorem}
\begin{proof} We set $g_i=t^{a_i}-t^{b_i}$ and
$h_i=t^{D(a_i)}-t^{D(b_i)}$ for $i=1,\ldots,m$. Notice that
$h_i$ is equal to $\widetilde{g}_i=g_i(t^{d_1},\ldots,t^{d_s})$, the
evaluation of 
$g_i$ at $(t_1^{d_1},\ldots,t_s^{d_s})$. 

$\Rightarrow$) By Lemma~\ref{dec19-11}, one has
the inclusion $(\widetilde{\mathcal{B}})\subset
I(\widetilde{\mathcal{L}})$. To show the reverse 
inclusion take a binomial $0\neq f\in I(\widetilde{\mathcal{L}})$. We
may assume that 
$f=t^{a^+}-t^{a^-}$. Then, by Lemma~\ref{dec19-11}, there is
$g=t^{c^+}-t^{c^-}$ in $I(\mathcal{L})$ such that $f=t^{D(c^+)}-t^{D(c^-)}$. By
hypothesis we can write $g=\sum_{i=1}^mf_ig_i$ for some
$f_1,\ldots,f_m$ in $S$. Then, evaluating both sides of this equality at
$(t_1^{d_1},\ldots,t_s^{d_s})$, we get
$$
f=t^{D(c^+)}-t^{D(c^-)}=g(t_1^{d_1},\ldots,t_s^{d_s})
=\sum_{i=1}^mf_i(t_1^{d_1},\ldots,t_s^{d_s})g_i(t_1^{d_1},\ldots,t_s^{d_s})=
\sum_{i=1}^m\widetilde{f}_ih_i,
$$
where $\widetilde{f}_i=f_i(t_1^{d_1},\ldots,t_s^{d_s})$ for all $i$. Then, $f\in
(\widetilde{\mathcal{B}})$. 

$\Leftarrow$) We may assume that $h_1,\ldots,h_r$ is a minimal set of
generators of $I(\mathcal{L})$ for some $r\leq m$. Consider the
$K$-vector spaces
$\widetilde{V}=I(\widetilde{\mathcal{L}})/\mathfrak{m}I(\widetilde{\mathcal{L}})$
and
$V=I(\mathcal{L})/\mathfrak{m}I(\mathcal{L})$, where
$\mathfrak{m}=(t_1,\ldots,t_s)$. Since the images, in
$\widetilde{V}$, of
$h_1,\ldots,h_r$ form a $K$-basis for $\widetilde{V}$, it follows
that the images, in $V$, of $g_1,\ldots,g_r$ are linearly 
independent. On the other hand, by Lemma~\ref{jul9-12}, the minimum
number of generators of $I(\mathcal{L})$ is also $r$ and
$r=\dim_K(V)$. Thus, the images, in $V$, of  $g_1,\ldots,g_r$ form a $K$-basis
for $V$. Consequently $\mathcal{B}$ generates $I(\mathcal{L})$ by
Nakayama's lemma (see \cite[Proposition~2.5.1 and
Corollary~2.5.2]{monalg}).
\end{proof}

\begin{definition} An ideal $I\subset S$ is called a {\it complete
intersection\/} if  
there exists  $g_1,\ldots,g_m$ such that $I=(g_1,\ldots,g_m)$, 
where $m$ is the height of $I$. 
\end{definition}

Recall that a graded binomial ideal $I\subset S$ is a complete
intersection if and only if $I$ is generated by a homogeneous regular 
sequence, consisting of binomials, with ${\rm ht}(I)$ elements 
(see for instance \cite[Proposition~1.3.17, Lemma~1.3.18]{monalg}).  

\begin{corollary}\label{ci-i(x)-p}  $I(\widetilde{\mathcal{L}})$ is
a complete intersection $($resp. Cohen-Macaulay, Gorenstein$)$ if 
and only if $I(\mathcal{L})$ is a complete intersection $($resp.
Cohen-Macaulay, Gorenstein$)$. 
\end{corollary}

\begin{proof} The rank of $\mathcal{L}$ is equal to the height of
$I(\mathcal{L})$ \cite[Proposition~7.5]{cca}. 
Since $D$ is non-singular, $\mathcal{L}$ and
$\widetilde{\mathcal{L}}=D(\mathcal{L})$ have the same rank. Thus, 
$I(\mathcal{L})$ and $I(\widetilde{\mathcal{L}})$
have the same height. Therefore, the result follows from
Theorem~\ref{dec22-11} and Lemma~\ref{jul9-12}. 
\end{proof}

\section{Lattice ideals of dimension $1$}\label{lattice-dim1-section}

We continue to use the notation and definitions used in
Section~\ref{ci-section}. In this section, we study the map $I\mapsto
\widetilde{I}$, when $I$ is a graded lattice ideal of dimension $1$.
In this case, we relate the degrees of $I$ and $\widetilde{I}$ and show a
formula for the regularity and the degree of $\widetilde{I}$ when $I$ is the
toric ideal of a monomial curve. We show that the regularity of the saturation of
certain one dimensional graded ideals is additive in a certain sense.

We begin by identifying some elements in the torsion subgroup of
$\mathbb{Z}^s/D(\mathcal{L})$. The proof of the following lemma is
straightforward.

\begin{lemma}\label{sep1-12-1} If $\mathcal{L}$ is a homogeneous
lattice, with respect 
to $d_1,\ldots,d_s$, of rank $s-1$, then for each $i,j$ there is 
a positive integer $\eta_{i,j}$ such that 
$\eta_{i,j}(d_je_i-d_ie_j)$ is in $\mathcal{L}$ and
$\eta_{i,j}(d_id_je_i-d_id_je_j)$ is in $D(\mathcal{L})$. In particular,
$e_i-e_j$ is in $T(\mathbb{Z}^s/D(\mathcal{L}))$ for any $i,j$. 
\end{lemma}

\begin{lemma}\label{aug22-12-x} Let $\mathcal{L}\subset\mathbb{Z}^{s}$
be a homogeneous lattice of rank $s-1$. If $\gcd(d_1,\ldots,d_s)=1$, 
then 
$$
|T(\mathbb{Z}^s/D(\mathcal{L}))|=
|\mathbb{Z}^s/D(\mathbb{Z}^s)|\, |T(\mathbb{Z}^s/\mathcal{L})|
=\det(D)|T(\mathbb{Z}^s/\mathcal{L})|.
$$
\end{lemma}

\begin{proof} The second equality follows 
from the equality $|\mathbb{Z}^s/D(\mathbb{Z}^s)|=\det(D)$. Next, we
show the first equality. Consider the following sequence of $\mathbb{Z}$-modules:
$$
0\rightarrow
T(\mathbb{Z}^s/\mathcal{L})\stackrel{\sigma}{\rightarrow}T(\mathbb{Z}^s/D(\mathcal{L}))
\stackrel{\rho}{\rightarrow}
\mathbb{Z}^s/D(\mathbb{Z}^s)\rightarrow 0,
$$
where $a+\mathcal{L}\stackrel{\sigma}{\mapsto}
D(a)+D(\mathcal{L})$ and $a+D(\mathcal{L})\stackrel{\rho}{\mapsto}
a+D(\mathbb{Z}^s)$. It suffices to show that this sequence is exact.
It is not hard to see that $\sigma$ is injective and that ${\rm
im}(\sigma)={\rm ker}(\rho)$.  Next, we show that
$\rho$ is onto. We need only show that $e_k+D(\mathbb{Z}^s)$ is
in the image of $\rho$ for $k=1,\ldots,s$. For simplicity of notation
we assume that $k=1$. There are integers $\lambda_1,\ldots,\lambda_s$
such that $1=\sum_i\lambda_id_i$. Then, using that $D(\mathbb{Z}^s)$
is generated by $d_1e_1,\ldots,d_se_s$, we obtain
\begin{eqnarray*}
e_1+D(\mathbb{Z}^s)&=&\lambda_1d_1e_1+\lambda_2d_2e_1+
\cdots+\lambda_sd_se_1+D(\mathbb{Z}^s)\\
&=&\lambda_1d_1e_1+\lambda_2d_2(e_1-e_2)+\cdots+
\lambda_sd_s(e_1-e_s)+D(\mathbb{Z}^s)\\
&=&\lambda_2d_2(e_1-e_2)+\cdots+\lambda_sd_s(e_1-e_s)+D(\mathbb{Z}^s).
\end{eqnarray*}
Hence, by Lemma~\ref{sep1-12-1}, the element
$\lambda_2d_2(e_1-e_2)+\cdots+\lambda_sd_s(e_1-e_s)+D(\mathcal{L})$ is a torsion
element of $\mathbb{Z}^s/D(\mathcal{L})$ that maps to
$e_1+D(\mathbb{Z}^s)$ under the map $\rho$.
\end{proof}

\begin{theorem}\label{dec5-12} If $I=I(\mathcal{L})$ is a graded
lattice ideal of dimension $1$, 
then 
$$ 
\deg(\widetilde{S}/\widetilde{I})=\frac{d_1\cdots
d_s}{\max\{d_1,\ldots,d_s\}}\deg(S/I).
$$
\end{theorem}

\begin{proof} We set $r=\gcd(d_1,\ldots,d_s)$ and 
$D'={\rm diag}(d_1/r,\ldots,d_s/r)$. As $I$ and $\widetilde{I}$ are graded
lattice ideals of dimension $1$, according to some results of 
\cite{prim-dec-critical} and \cite{degree-lattice}, one has
$$
\deg({S}/{I})=\frac{\max\{d_1,\ldots,d_s\}}{r}|T(\mathbb{Z}^s/\mathcal{L})|\ 
\mbox{ and }\ 
\deg(\widetilde{S}/\widetilde{I})=
|T(\mathbb{Z}^s/\widetilde{\mathcal{L}})|,
$$
where $\widetilde{\mathcal{L}}=D(\mathcal{L})$. 
Hence, by Lemma~\ref{aug22-12-x}, we get 
\begin{eqnarray*}
\deg(\widetilde{S}/\widetilde{I})&=&|T(\mathbb{Z}^s/D(\mathcal{L}))|=
r^{s-1}|T(\mathbb{Z}^s/D'(\mathcal{L}))|\\
&=& r^{s-1}
|\mathbb{Z}^s/D'(\mathbb{Z}^s)|\, |T(\mathbb{Z}^s/\mathcal{L})|
=r^{s-1}\det(D')|T(\mathbb{Z}^s/\mathcal{L})|\\
&=&\frac{d_1\cdots d_s}{r}|T(\mathbb{Z}^s/\mathcal{L})|=\frac{d_1\cdots
d_s}{\max\{d_1,\ldots,d_s\}}\deg(S/I).
\end{eqnarray*}
The second equality can be shown using that the order of
$T(\mathbb{Z}^s/D(\mathcal{L}))$ is the $\gcd$ of all $s-1$ minors of
a presentation matrix of $\mathbb{Z}^s/D(\mathcal{L})$. 
\end{proof}

\begin{definition}\label{frobenius-number-def} If $\mathcal{S}$ is a
numerical semigroup of 
$\mathbb{N}$, the {\it Frobenius number\/} 
of $\mathcal{S}$, denoted by $g(\mathcal{S})$, is the largest 
integer not in $\mathcal{S}$. 
\end{definition}

The next result gives an explicit formula for the
regularity in terms of Frobenius numbers, that can be used to
compute the regularity using some available algorithms (see the
monograph \cite{frobenius-problem}). Using {\it Macaulay\/}$2$ 
\cite{mac2}, we can use this formula to compute the Frobenius number
of the corresponding semigroup.

\begin{theorem}\label{reg-deg-mcurves} If $I$ is the toric ideal of
$K[y_1^{d_1},\ldots,y_1^{d_s}]\subset K[y_1]$ and $r=\gcd(d_1,\ldots,d_s)$, then 

{\rm (a)} $\textstyle{\rm reg}(\widetilde{S}/\widetilde{I})=
\textstyle r\cdot g(\mathcal{S})+1+\sum_{i=1}^s(d_i-1)$, where
$\mathcal{S}=\mathbb{N}(d_1/r)+\cdots+\mathbb{N}(d_s/r)$.

{\rm (b)} $\deg(\widetilde{S}/\widetilde{I})=d_1\cdots d_s/r$. 
\end{theorem}

\begin{proof} (a): We set $d_i'=d_i/r$ for $i=1,\ldots,s$. Let $L$ be
the toric ideal of 
$K[y_1^{d_1'},\ldots,y_1^{d_s'}]$ and let $\widetilde{L}$ be the
homogenization of $L$ with respect to $t_1^{d_1'},\ldots,t_s^{d_s'}$.
It is not hard to see that the toric ideals $I$ and $L$ are equal. 
Let $F_L(t)$ be the Hilbert series of $S/L$,
where $S$ has the grading induced by setting $\deg(t_i)=d_i'$
for all $i$. As $\gcd(d_1',\ldots,d_s')=1$, $\mathcal{S}$ is a
numerical semigroup, and by \cite[Remark~4.5, p.~200]{stcib} we can write
$F_L(t)=f(t)/(1-t)$, where $f(t)$ is a polynomial in $\mathbb{Z}[t]$ of degree
$g(\mathcal{S})+1$. Then, by Theorem~\ref{jul10-12}, we get
$$
\textstyle{\rm reg}(S/L)=\deg(F_L(t))-{\rm
ht}(L)+\sum_{i=1}^sd_i/r=g(\mathcal{S})-(s-1)+\sum_{i=1}^sd_i/r.
$$
Notice that $\widetilde{I}$ and $\widetilde{L}$ are Cohen-Macaulay
lattice ideals of height $s-1$. Since $\widetilde{I}$ is the
homogenization of $\widetilde{L}$ with 
respect to $t_1^r,\ldots,t_s^r$, by Lemma~\ref{neves-lemma}(b) and
Theorem~\ref{aug23-12}, we get
\begin{eqnarray*}
\textstyle{\rm reg}(\widetilde{S}/\widetilde{I})&=&
\textstyle (s-1)(r-1)+r\cdot{\rm reg}(\widetilde{S}/\widetilde{L})=
(s-1)(r-1)+r\cdot{\rm reg}({S}/{L})\\
& =& (s-1)(r-1)+r\left(\textstyle
g(\mathcal{S})-(s-1)+\sum_{i=1}^sd_i/r\right)\\
&=& \textstyle r\cdot g(\mathcal{S})+1+\sum_{i=1}^s(d_i-1).
\end{eqnarray*}

(b): By a result of \cite{prim-dec-critical},
$\deg(S/I)=\max\{d_1,\ldots,d_s\}/r$. Hence, the formula follows 
from Theorem~\ref{dec5-12}.
\end{proof}

\begin{definition} The {\it graded reverse lexicographical order\/}
(GRevLex for short) 
is defined as $t^b\succ t^a$ if and only if $\deg(t^b)>\deg(t^a)$ or
$\deg(t^b)=\deg(t^a)$ and the last  nonzero entry of $b-a$ is
negative. The {\it reverse lexicographical order\/} (RevLex for short)
is defined as $t^b\succ t^a$ if and only if 
the last  nonzero entry of $b-a$ is
negative.  
\end{definition}

\begin{corollary}\label{nov30-12} 
Let $\succ$ be the RevLex order. 
If $I$ is a graded lattice ideal and $\dim(S/I)=1$, then   
$$
{\rm reg}(S/I)={\rm reg}(\widetilde{S}/\widetilde{I})=
{\rm reg}(\widetilde{S}/{\rm in}(\widetilde{I}))
={\rm reg}(S/{\rm in}(I)),
$$
where ${\rm in}({I})$, ${\rm in}(\widetilde{I})$ are the initial
ideals of $I$, $\widetilde{I}$, with respect to $\succ$, respectively. 
\end{corollary}

\begin{proof} The quotients rings $\widetilde{S}/\widetilde{I}$ and
$\widetilde{S}/{\rm in}(\widetilde{I})$ are 
Cohen-Macaulay standard algebras of dimension $1$ because $t_s$ is a
regular element of 
both rings. Hence, these two rings have the same Hilbert function and
the same index of regularity. Therefore ${\rm
reg}(\widetilde{S}/\widetilde{I})$ is equal to 
${\rm reg}(\widetilde{S}/{\rm in}(\widetilde{I}))$. 
As ${\rm in}(\widetilde{I})=\widetilde{{\rm in}(I)}$,  
by Theorem~\ref{aug23-12}, we get 
$$
{\rm reg}(S/I)={\rm reg}(\widetilde{S}/\widetilde{I})={\rm
reg}(\widetilde{S}/{\rm in}(\widetilde{I}))=
{\rm
reg}(\widetilde{S}/\widetilde{{\rm in}(I)})=
{\rm reg}(S/{\rm in}(I)),
$$
as required. 
\end{proof}

\begin{theorem}\label{vila-12-12-12} 
Let $V_1,\ldots,V_c$ be a partition of 
$V=\{t_1,\ldots,t_s\}$ and let $\ell$ be a positive integer. Suppose
that $K[V_k]$ and $K[V]$ are polynomial rings with the standard
grading for $k=1,\ldots,c$. If $I_k$
is a graded binomial ideal of $K[V_k]$ such that $t_i^\ell-t_j^\ell\in I_k$ for
$t_i,t_j\in V_k$ and $\mathcal{I}$ is the ideal of $K[V]$ generated by
all binomials $t_i^{\ell}-t_j^{\ell}$ with $1\leq i,j\leq s$, then 
\begin{eqnarray*}
(\mathrm{i})& & (I_1+\cdots+I_c+\mathcal{I}\colon h^\infty)=
\textstyle(I_1\colon h_1^\infty)+\cdots+
\textstyle(I_c\colon h_c^\infty)+\mathcal{I}, \mbox{ and }
\\
(\mathrm{ii})& &
{\rm reg}\, K[V]/(I_1+\cdots+I_c+\mathcal{I}\colon
h^\infty)=
 \sum_{k=1}^c{\rm reg}\,
K[V_k]/(I_k\colon h_k^\infty)+(c-1)(\ell-1),
\end{eqnarray*}
where $h=t_1\cdots t_s$ and $h_k=\prod_{t_i\in V_k}t_i$ for
$k=1,\ldots,c$.
\end{theorem}

\begin{proof} The proofs of (i) and (ii) are by induction on $c$. If
$c=1$ the asserted equalities hold because in this case
$\mathcal{I}\subset I_1$. We set 
\begin{eqnarray*}
J&=&(I_1+\cdots+I_c+\mathcal{I}\colon h^\infty),\ \ J_k=(I_k\colon
h_k^\infty),
 \\ 
L&=&I_1+\cdots+I_{c-1}+\mathcal{I}',\mbox{ where } 
\mathcal{I}'=(\{t_i^\ell-t_j^\ell\vert\, t_i,t_j\in V'\})\mbox{ and }
V'=V_1\cup\cdots\cup
V_{c-1},
\end{eqnarray*}
and $g=\prod_{t_i\in V'}t_i$. First, we show the case $c=2$. (i): We set 
$J'=J_1+J_2+\mathcal{I}$. Clearly one
has the inclusion $J'\subset J$. To show the reverse
inclusion it suffices to show that $J'$ 
is a lattice ideal. Since this ideal is graded of dimension $1$ and
$V(J',t_i)=0$ for $i=1,\ldots,s$, we need only show that $J'$ 
is Cohen-Macaulay (see Lemma~\ref{apr24-12-1}). As $J_1, J_2$ are lattice
ideals of dimension $1$, they are Cohen-Macaulay. Hence, $J_1+J_2$ is
Cohen-Macaulay of dimension $2$ (see \cite[Lemma~4.1]{Vi2}). Pick
$t_i\in V_1$ and $t_j\in V_2$. The binomial $f=t_i^\ell-t_j^\ell$ is
regular modulo $J_1+J_2$. Indeed, if $f$ is in some associated prime
$\mathfrak{p}$ of $J_1+J_2$, then $\mathcal{I}\subset \mathfrak{p}$
and consequently the height of $\mathfrak{p}$ is at least $s-1$, a
contradiction because all associated primes of $J_1+J_2$ have height
$s-2$. Hence, since $J_1+J_2+\mathcal{I}$ is equal to $J_1+J_2+(f)$,
the ideal $J'$ is Cohen-Macaulay. (ii): There is an exact sequence 
$$
0\longrightarrow (K[V]/(J_1+J_2))[-\ell]\stackrel{f}{\longrightarrow}
K[V]/(J_1+J_2)\longrightarrow K[V]/J\longrightarrow 0.
$$
Hence, by the additivity of Hilbert series, 
we get
$$
H_J(t)=H_{(J_1+J_2)}(t)(1-t^\ell)=H_{J_1}(t)H_{J_2}(t)(1-t^\ell),
$$
where $H_J(t)$ is the Hilbert series of $K[V]/J$ 
(cf. arguments below). 
The ideals
$J,J_1,J_2$ are graded lattice ideals of dimension $1$, hence they are
Cohen-Macaulay. Therefore (cf. arguments below) we obtain 
$$
{\rm reg}\, K[V]/J={\rm reg}\, K[V_1]/J_1+{\rm reg}\, K[V_2]/J_2+
(\ell-1),
$$
which gives the formula for the regularity. 

Next, we show the general case. (i): Applying the 
case $c=2$, by induction, we get  
\begin{eqnarray*}
J&=&(L+I_c+\mathcal{I}\colon h^\infty)=(L\colon g^\infty)+(I_c\colon
h_c^\infty)+\mathcal{I}\\ 
&=&[(I_1\colon h_1^\infty)+\cdots+(I_{c-1}\colon
h_{c-1}^\infty)+\mathcal{I}']+ (I_c\colon
h_c^\infty)+\mathcal{I}.
\end{eqnarray*}
Thus, $J=(I_1\colon h_1^\infty)+\cdots+(I_c\colon
h_c^\infty)+\mathcal{I}$, as required. (ii): We set 
$Q_1=(L\colon g^\infty)$ and $Q=Q_1+J_c$. 
From the isomorphism
$$
K[V]/Q\simeq (K[V']/Q_1)\otimes_K(
K[V_c]/J_c),
$$
we get that $H_Q(t)=H_1(t)H_2(t)$, where $H_Q(t)$ is the
Hilbert series of $K[V]/Q$ and $H_1(t)$,
$H_2(t)$ are the Hilbert series of $K[V']/Q_1$ and $K[V_c]/J_c$, 
respectively. Since $Q_1$ and $J_c$ are graded lattice ideals of
dimension $1$, they are Cohen-Macaulay and their 
Hilbert series can be written as $H_i(t)=f_i(t)/(1-t)$, with 
$f_i(t)\in\mathbb{Z}[t]$ for $i=1,2$ and such that $\deg(f_1)={\rm
reg}\, K[V']/Q_1$ 
and $\deg(f_2)={\rm reg}\, K[V_c]/J_c$. Fix $t_i\in V'$ and $t_j\in V_c$. If
$f=t_i^\ell-t_j^\ell$, then by the case $c=2$ there is an exact
sequence 
$$
0\longrightarrow (K[V]/Q)[-\ell]\stackrel{f}{\longrightarrow}
K[V]/Q\longrightarrow K[V]/J\longrightarrow 0.
$$
Hence, by the additivity of Hilbert series, we get
$$
H_J(t)=H_Q(t)(1-t^\ell)=H_1(t)H_2(t)(1-t^\ell)=
\frac{f_1(t)f_2(t)(1+t+\cdots+t^{\ell-1})}{(1-t)}.
$$
Therefore, as $J$ is a graded lattice ideal of dimension $1$, by induction we
get 
\begin{eqnarray*} 
{\rm reg}\, K[V]/J&=&{\rm reg}\, K[V']/Q_1+{\rm reg}\, K[V_c]/J_c+ (\ell-1)\\ 
&=& \left[\textstyle  \sum_{k=1}^{c-1}{\rm reg}\, K[V_k]/(I_k\colon
g_k^\infty)+(c-2)(\ell-1)\right]+{\rm reg}\, K[V_c]/J_c+
(\ell-1),
\end{eqnarray*}
which gives the formula for the regularity. 
\end{proof}

\section{Vanishing ideals}\label{applications-to-i(x)}

In this section we study graded vanishing ideals over arbitrary fields and give
some applications of the results of Section~\ref{lattice-dim1-section}. For
finite fields, we give formulae for the degree and regularity of
graded vanishing
ideals over degenerate tori. Given a
sequence of positive integers, we construct vanishing ideals, over
finite fields, with prescribed regularity and degree of a certain
type. We characterize when a graded lattice ideal of dimension $1$ is a
vanishing ideal in terms of the degree. We show that the vanishing 
ideal of $X$ is a lattice ideal of dimension $1$ if and only if $X$ is
a finite subgroup of a projective torus. For finite fields, it is
shown that $X$ is a  subgroup of a projective torus if and only if $X$
is parameterized by monomials. 

Let $K\neq\mathbb{F}_2$ be a field and let
$\mathbb{P}^{s-1}$ be a 
projective space of dimension $s-1$ over the field $K$. If $X$ is a subset of
$\mathbb{P}^{s-1}$, the {\it vanishing ideal\/} of
$X$,  denoted by
$I(X)$, is the ideal 
of $\widetilde{S}$ generated by all the homogeneous polynomials that vanish on $X$. 

\begin{definition} Given a sequence
$v=(v_1,\ldots,v_s)$ of  
positive integers, the set 
$$
\{[(x_1^{v_{1}},\ldots,x_s^{v_s})]\, \vert\, x_i\in K^*\mbox{ for all
}i\}\subset\mathbb{P}^{s-1}
$$
is called a {\it degenerate projective torus\/} of type $v$, 
where $K^*=K\setminus\{0\}$. If $v_i=1$ for all $i$, this set 
is called a {\it projective torus\/} in $\mathbb{P}^{s-1}$ and it is
denoted by $\mathbb{T}$.   
\end{definition}

The next result was shown in \cite{ci-mcurves} under the hypothesis that 
$I(X)$ is a complete intersection.  

\begin{corollary}\label{dec9-12} Let $K=\mathbb{F}_q$ be a finite field and let $X$
be a degenerate projective torus of type $v=({v_1},\ldots,v_s)$. If
$d_i=(q-1)/\gcd(v_i,q-1)$ for 
$i=1,\ldots,s$ and $r=\gcd(d_1,\ldots,d_s)$, then 
$$
\textstyle{\rm
reg}(\widetilde{S}/I(X))=r\cdot g(\mathcal{S})+1+\sum_{i=1}^s(d_i-1)\
\mbox{ and }\ \deg(\widetilde{S}/I(X))=d_1\cdots d_s/r,
$$
where $\mathcal{S}=\mathbb{N}(d_1/r)+\cdots+\mathbb{N}(d_s/r)$ is the
semigroup generated by 
$d_1/r,\ldots,d_s/r$. 
\end{corollary} 

\begin{proof} Let $I$ be the toric ideal of
$K[y_1^{d_1},\ldots,y_1^{d_s}]$ and let $\widetilde{I}$ be the
homogenization of $I$ with respect to $t_1^{d_1},\ldots,t_s^{d_s}$. 
According to \cite[Lemma~3.1]{ci-mcurves}, $\widetilde{I}$ is equal to
$I(X)$. Hence, the result follows from Theorem~\ref{reg-deg-mcurves}.
\end{proof}

\begin{lemma}\label{renteria-proof} Given positive integers
$d_1,\ldots,d_s$, there is a prime number 
$p$ such that $d_i$ divides $p-1$ for all $i$.
\end{lemma}

\begin{proof} We set $m={\rm lcm}(d_1,\ldots,d_s)$ and $a=1$. As $a$ and $m$ are
relatively prime positive integers, by a classical theorem of
Dirichlet \cite[p.~25, p~.61]{serre-book}, there exist infinitely many
primes $p$ such that $p\equiv a\ \mod (m)$. Thus, we can write $p-1=km$ for
some integer $k$.  This proves that $d_i$ divides $p-1$ for all $i$.
\end{proof}

The next result allows us to
construct vanishing ideals over finite fields with prescribed regularity and degree of a
certain type.

\begin{proposition}\label{dec30-12-3} Given a sequence $d_1,\ldots,d_s$ of positive
integers, there is a finite field $K=\mathbb{F}_q$ and a degenerate
projective torus $X$ such that
$$
\textstyle{\rm
reg}(\widetilde{S}/I(X))=r\cdot g(\mathcal{S})+1+\sum_{i=1}^s(d_i-1)\
\mbox{ and }\ \deg(\widetilde{S}/I(X))=d_1\cdots d_s/r,
$$
where $r=\gcd(d_1,\ldots,d_s)$ and 
$\mathcal{S}=\mathbb{N}(d_1/r)+\cdots+\mathbb{N}(d_s/r)$.  
\end{proposition}

\begin{proof} 
By Lemma~\ref{renteria-proof}, there is a prime number $q$ such that 
$d_i$ divides $q-1$ for $i=1,\ldots,s$. We set $K=\mathbb{F}_q$ and
$v_i=(q-1)/d_i$ for all
$i$. If $X$ is a degenerate torus of type $v=(v_1,\ldots,v_s)$ then,
by Corollary~\ref{dec9-12}, the result follows.
\end{proof}

\begin{proposition}\label{dec30-12-2} Let $L$ be a graded lattice ideal of
$\widetilde{S}$ of dimension $1$ over an arbitrary field $K$ and let
$L=\mathfrak{q}_1\cap\cdots\cap\mathfrak{q}_m$ be a minimal primary 
decomposition of $L$. Then, $\deg(\widetilde{S}/L)\geq m$ with equality
if and only if $L=I(X)$ for some finite set $X$ of a 
projective torus $\mathbb{T}$ of $\mathbb{P}^{s-1}$.
\end{proposition}

\begin{proof} The inequality $\deg(\widetilde{S}/L)\geq m$ follows at
once from Proposition~\ref{additivity-of-the-degree}. Assume that
$\deg(\widetilde{S}/L)=m$. Let $\mathfrak{q}=\mathfrak{q}_i$ be any primary
component of $L$. Then, $\deg(\widetilde{S}/\mathfrak{q})=1$.
Consider  the reduced Gr\"obner
basis $\mathcal{G}=\{g_1,\ldots,g_p\}$ of $\mathfrak{q}$ relative to the 
graded reverse lexicographical order
of $\widetilde{S}$. As usual, we denote the initial term of $g_i$ by
${\rm in}(g_i)$. As the degree and the Krull dimension of
$\widetilde{S}/\mathfrak{q}$  
are equal to $1$, $H_\mathfrak{q}(d)=1$ for $d\gg 0$, i.e.,
$\dim_K(\widetilde{S}/\mathfrak{q})_d=1$ for $d\gg 0$. Using that $t_i$
is not a zero divisor of $\widetilde{S}/\mathfrak{q}$ for
$i=1,\ldots,s$, we get that
$t_s$ does not divide ${\rm in}(g_i)$ for any $i$. Then, $t_s^d+\mathfrak{q}$
generates $(\widetilde{S}/\mathfrak{q})_d$ as a $K$-vector space for
$d\gg 0$. Hence, for $i=1,\ldots,s-1$, there is $\mu_i\in K^*$ such that
$t_it_s^{d-1}-\mu_i t_s^d\in\mathfrak{q}$. Thus, $t_it_s^{d-1}$ is in
the initial ideal ${\rm in}(\mathfrak{q})$ of $\mathfrak{q}$ 
which is generated by ${\rm in}(g_1),\ldots,{\rm
in}(g_p)$. In particular, $t_i\in{\rm in}(L)$ for $i=1,\ldots,s-1$ 
and $p=s-1$ because $\mathcal{G}$ is reduced. 
Therefore for $i=1,\ldots,s-1$, using that $\mathcal{G}$ is a reduced Gr\"obner 
basis, we may assume that $g_i=t_i-\lambda_it_s$ for some
$\lambda_i\in K^*$. If $Q=(\lambda_1,\ldots,\lambda_{s-1},1)$, it is
seen that $\mathfrak{q}$ is the vanishing ideal of $[Q]$. Let
$X=\{[Q_1],\ldots,[Q_m]\}$ be the set of points in the projective
torus $\mathbb{T}\subset\mathbb{P}^{s-1}$ such that
$\mathfrak{q}_i$ is the vanishing ideal of $[Q_i]$, then 
$$
I(X)=\cap_{i=1}^mI_{[Q_i]}=\mathfrak{q}_1\cap\cdots\cap\mathfrak{q}_m=L,
$$
where $I_{[Q_i]}$ is the vanishing ideal of $[Q_i]$.

To show the converse, assume that $L$ is the vanishing ideal of a
finite set of points $X$ in a projective torus
$\mathbb{T}$. Let $[Q]=[(\alpha_i)]$ be a point in $X$ and let $I_{[Q]}$ be the 
vanishing ideal of $[Q]$. It is not hard to see that the ideal $I_{[Q]}$ is given by
\begin{equation*}
I_{[Q]}=(\alpha_1t_2-\alpha_2t_1,\alpha_1t_3-\alpha_3t_1,\ldots,
\alpha_1t_s-\alpha_st_1).
\end{equation*}
The primary decomposition of $L=I(X)$ is 
$I(X)=\cap_{[Q]\in X}I_{[Q]}$ because $I_{[Q]}$ is
a prime ideal of $\widetilde{S}$ for any $[Q]\in X$. To complete the proof notice
that $\deg(\widetilde{S}/I_{[Q]})=1$ for any $[Q]\in X$ and
$\deg(\widetilde{S}/I(X))=|X|$. 
\end{proof}

A similar statement holds for non-graded lattice ideals of dimension $0$.

\begin{proposition} Let $L$ be a lattice ideal of
$S$ of dimension $0$ and let
$L=\mathfrak{q}_1\cap\cdots\cap\mathfrak{q}_m$ be a minimal primary 
decomposition of $L$. Then, $\deg(S/L)\geq m$ with equality
if and only if $L=I(X^*)$ for some finite set $X^*$ contained in 
an affine torus $\mathbb{T}^*$ of $K^s$.
\end{proposition}

Let $K$ be a field with $K\neq\mathbb{F}_2$ and 
let $y^{v_1},\ldots,y^{v_s}$ be a finite set of monomials.  
As usual if $v_i=(v_{i1},\ldots,v_{in})\in\mathbb{N}^n$, 
then we set 
$$
y^{v_i}=y_1^{v_{i1}}\cdots y_n^{v_{in}},\ \ \ \ i=1,\ldots,s,
$$
where $y_1,\ldots,y_n$ are the indeterminates of a ring of 
polynomials with coefficients in $K$. The {\it projective algebraic
toric set} parameterized by $y^{v_1},\ldots,y^{v_s}$ is the set:
$$
\{[(x_1^{v_{11}}\cdots x_n^{v_{1n}},\ldots,x_1^{v_{s1}}\cdots
x_n^{v_{sn}})]\, \vert\, x_i\in K^*\mbox{ for all
}i\}\subset\mathbb{P}^{s-1}.
$$
A set of this form is said to be {\it parameterized by monomials\/}. 

\begin{proposition}\label{dec30-12-1} Let $K\neq\mathbb{F}_2$ be an
arbitrary field and let 
$X\subset\mathbb{P}^{s-1}$. Then the following hold.
\begin{itemize}
\item[\rm(a)] $I(X)$ is a lattice ideal of dimension $1$ if and only
if $X$ is a finite subgroup of $\mathbb{T}$.
\item[\rm(b)] If $K$ is finite, then $X$ is a
subgroup of $\mathbb{T}$ if and only if $X$ is parameterized by
monomials. 
\end{itemize}
\end{proposition}
\begin{proof} (a): ($\Rightarrow$) The set $X$ is finite because
$\dim\, \widetilde{S}/I(X)=1$ (see \cite[Proposition~6, p.~441]{CLO}). Let
$[\alpha]=[(\alpha_i)]$ be a point of $X$ and let $I_{[\alpha]}$ be
its vanishing ideal. We may assume that
$\alpha_k=1$ for some $k$. Since the ideal 
\begin{equation*}\label{tag-vanishing-point}
I_{[\alpha]}=(t_1-\alpha_1t_k,\ldots,
t_{k-1}-\alpha_{k-1}t_k,t_{k+1}-\alpha_{k+1}t_k,\ldots,t_s-\alpha_st_k)\tag{$\ast$}
\end{equation*}
is a minimal prime of $I(X)$, $\alpha_i\neq 0$ for all $i$ because
$t_i$ is not a zero divisor of $\widetilde{S}/I(X)$. Thus, $[\alpha]\in\mathbb{T}$. This
proves that $X\subset\mathbb{T}$. Next, we show that $X$ is a subgroup
of $\mathbb{T}$. Let $g_1,\ldots,g_r$ be a generating set of $I(X)$
consisting of binomials and let
$[\alpha]=[(\alpha_i)],[\beta]=[(\beta_i)]$ be two elements
of $X$. We set $\gamma=\alpha\cdot\beta=(\alpha_i\beta_i)$. Since the
entries of $\gamma$ are all non-zero, we may assume 
that $\gamma_s=1$. Since $g_i(\alpha)=0$ and $g_i(\beta)=0$ for all $i$, we get
that $g_i(\gamma)=0$ for all $i$. Hence, $I(X)\subset I_{[\gamma]}$
and consequently $I_{[\gamma]}$ is a minimal prime of $I(X)$. Hence
there is $[\gamma']\in X$, with $\gamma_s'=1$ such that
$I_{[\gamma]}=I_{[\gamma']}$. It follows that $\gamma=\gamma'$. Thus,
$[\gamma]\in X$. By a similar argument it follows that
$[\alpha]^{-1}=[(\alpha_i^{-1})]$ is in $X$.  

(a): ($\Leftarrow$) The ideal $I(X)$ is
generated by binomials, this follows from
\cite[Proposition~2.3(a)]{EisStu} and its proof. Since $I(X)$ is equal
to $\cap_{[\alpha]\in X}I_{[\alpha]}$, using
Eq.~(\ref{tag-vanishing-point}), we get that $t_i$ is not a zero
divisor of $S/I(X)$ for all $i$. Hence, by Theorem~\ref{jun12-02}, $I(X)$ is a
lattice ideal. 

(b): ($\Rightarrow$) By the fundamental theorem of finitely generated
abelian groups, $X$ is a direct product of cyclic groups. Hence,
there are $[\alpha_1],\ldots,[\alpha_n]$ in $X$ such that
$$
X=\left\{\left.[\alpha_1]^{i_1}\cdots[\alpha_n]^{i_n}\, \right\vert\,
i_1,\ldots,i_n\in\mathbb{Z}\right\}.
$$
If $\beta$ is a generator of $(K^*,\cdot\, )$, we can write
$$
\alpha_1=(\beta^{v_{11}},\ldots,\beta^{v_{s1}}),\ldots,
\alpha_n=(\beta^{v_{1n}},\ldots,\beta^{v_{sn}})
$$ 
for some $v_{ij}$'s in $\mathbb{N}$. Then, $[\gamma]$ is in $X$ if and
only if we can write
$$
[\gamma]=[(
(\beta^{i_1})^{v_{11}}\cdots (\beta^{i_n})^{v_{1n}},\ldots,
(\beta^{i_1})^{v_{s1}}\cdots (\beta^{i_n})^{v_{sn}}
)]
$$
for some $i_1,\ldots,i_n\in\mathbb{Z}$. Therefore, $X$ is parameterized
by the monomials $y^{v_1},\ldots,y^{v_s}$, where $v_i=(v_{i1},\ldots,v_{is})$ for
$i=1,\ldots,s$. 

(b): ($\Leftarrow$) If $X\subset\mathbb{P}^{s-1}$ is a projective
algebraic toric set parameterized by 
$y^{v_1},\ldots,y^{v_s}$, then by the exponent laws it is not hard to show that $X$ is a
multiplicative group under componentwise multiplication. 
\end{proof}

The next structure theorem allows us---with the help of {\em
Macaulay\/}$2$ \cite{mac2}---to compute the vanishing ideal of 
an algebraic toric set parameterized by monomials over a 
finite field.

\begin{theorem}{\rm\cite[Theorem~2.1]{algcodes}}\label{ipn-ufpe-cinvestav}
Let $B=K[t_1,\ldots,t_s,y_1,\ldots,y_n,z]$ be a polynomial ring
over the finite field $K=\mathbb{F}_q$ and let $X$ be the 
algebraic toric set parameterized by $y^{v_1},\ldots,y^{v_s}$. Then 
$$
I(X)=(\{t_i-y^{v_i}z\}_{i=1}^s\cup\{y_i^{q-1}-1\}_{i=1}^n)\cap \widetilde{S}
$$
and $I(X)$ is a Cohen-Macaulay radical lattice ideal of 
dimension $1$.
\end{theorem}

The following theorem takes care of the infinite field case.

\begin{theorem}\label{ipn-ufpe-cinvestav-1}
Let $B=K[t_1,\ldots,t_s,y_1,\ldots,y_n,z]$ be a polynomial ring
over an infinite field $K$. If $X$ is an algebraic toric set
parameterized by monomials $y^{v_1},\ldots,y^{v_s}$, 
then 
$$
I(X)=(\{t_i-y^{v_i}z\}_{i=1}^s)\cap \widetilde{S} 
$$
and $I(X)$ is the toric ideal of $K[x^{v_1}z,\ldots,x^{v_s}z]$.
\end{theorem}

\begin{proof} We set
$I'=(t_1-y^{v_1}z,\ldots,t_s-y^{v_s}z)\subset B$. First we show the
inclusion $I(X)\subset I'\cap \widetilde{S}$.  
Take a homogeneous polynomial $F=F(t_1,\ldots,t_s)$ of degree $d$
that vanishes on $X$. We can write
\begin{equation}\label{aug27-09}
F=\lambda_1 t^{m_1}+\cdots+\lambda_r t^{m_r}\ \ \ \ (\lambda_i\in
K^*;\, m_i\in \mathbb{N}^s),
\end{equation}
where $\deg(t^{m_i})=d$ for all $i$. Write
$m_i=(m_{i1},\ldots,m_{is})$ for $1\leq i\leq r$. Applying the
binomial theorem to expand the right hand side of the equality
$$
t_j^{m_{ij}}=\left[(t_j-y^{v_j}z)+y^{v_j}z\right]^{m_{ij}},\ \ \
1\leq i\leq r,\ 1\leq j\leq s,
$$
and then substituting all the $t_j^{m_{i j}}$ in
Eq.~(\ref{aug27-09}), we obtain that $F$ can be written as:
\begin{equation}\label{23-jul-10}
F=\sum_{i=1}^sg_i(t_i-y^{v_i}z)+z^dF(y^{v_1},\ldots,y^{v_s})
=\sum_{i=1}^sg_i(t_i-y^{v_i}z)+z^dG(y_1,\ldots,y_n)
\end{equation}
for some $g_1,\ldots,g_s$ in $B$. Thus to show that $F\in I'\cap
S$ we need only show that $G=0$. We claim that $G$ vanishes on
$(K^*)^n$. Take an arbitrary sequence $x_1,\ldots,x_n$ of elements
of $K^*$. Making $t_i=x^{v_i}$ for all $i$ in
Eq.~(\ref{23-jul-10}) and using that $F$ vanishes on $X$, we
obtain
\begin{equation}\label{23-jul-10-3}
0=F(x^{v_1},\ldots,x^{v_s})=\sum_{i=1}^sg_i'(x^{v_i}-y^{v_i}z)+
z^dG(y_1,\ldots,y_n).
\end{equation}
We can make $y_i=x_i$ for all $i$ and $z=1$ in
Eq.~(\ref{23-jul-10-3}) to get that $G$ vanishes on
$(x_1,\ldots,x_n)$. This completes the proof of the claim.
Therefore $G$ vanishes on $(K^*)^n$ and since the field $K$ is
infinite it follows that $G=0$.

Next we show the inclusion $I(X)\supset I'\cap \widetilde{S}$. Let
$\mathcal{G}$ be a Gr\"obner basis of $I'$ with respect to the
lexicographical order $y_1\succ\cdots\succ y_n\succ z\succ
t_1\succ\cdots\succ t_s$. By Buchberger's algorithm \cite[Theorem~2,
p.~89]{CLO} the set $\mathcal{G}$ consists of binomials and by
elimination theory \cite[Theorem~2, p.~114]{CLO} the set
$\mathcal{G}\cap \widetilde{S}$ is a Gr\"obner basis of $I'\cap \widetilde{S}$. Hence
$I'\cap \widetilde{S}$ is a binomial ideal. Thus to show the inclusion
$I(X)\supset I'\cap \widetilde{S}$ it suffices to show that any binomial in
$I'\cap \widetilde{S}$ is homogeneous and vanishes on $X$. Take a binomial
$f=t^a-t^b$ in $I'\cap \widetilde{S}$, where $a=(a_i)$ and $b=(b_i)$ are in
$\mathbb{N}^s$. Then we can write
\begin{equation}\label{sept1-09}
f=\sum_{i=1}^sg_i(t_i-y^{v_i}z)
\end{equation}
for some polynomials $g_1,\ldots,g_s$ in $B$. Making $y_i=1$ for
$i=1,\ldots,n$ and $t_i=y^{v_i}z$ for $i=1,\ldots,s$, we get
$$
z^{a_1}\cdots z^{a_s}-z^{b_1}\cdots z^{b_s}=0\ \Longrightarrow\
a_1+\cdots+a_s=b_1+\cdots+b_s.
$$
Hence $f$ is homogeneous. Take a point $[P]$ in $X$ with
$P=(x^{v_1},\ldots,x^{v_s})$. Making $t_i=x^{v_i}$ in
Eq.~(\ref{sept1-09}), we get
$$
f(x^{v_1},\ldots,x^{v_s})=\sum_{i=1}^sg_i'(x^{v_i}-y^{v_i}z).
$$
Hence making $y_i=x_i$ for all $i$ and $z=1$, we get that
$f(P)=0$. Thus $f$ vanishes on $X$. Thus, we have shown the
equality $I(X)=I'\cap \widetilde{S}$. 

By \cite[Proposition 7.1.9]{monalg}
$I(X)$ is the toric ideal of $K[x^{v_1}z,\ldots,x^{v_s}z]$.
 \end{proof}

\section{Vanishing ideals over graphs}\label{applications-to-i(x)-1}

In this section, we study graded vanishing ideals over bipartite
graphs. For a projective algebraic toric set parameterized by the
edges of a bipartite graph, we are able to express the regularity of
the vanishing ideal in terms of the corresponding regularities for
the blocks of the graph. For bipartite graphs, we introduce a method that can be
used to compute the regularity. 

Let $K=\mathbb{F}_q$ be a finite field with $q$ elements and let 
$G$ be a simple graph with vertex set $V_G=\{y_1,\ldots,y_n\}$
and edge set $E_G$. We refer to \cite{Boll} for the general theory of
graphs. 

\begin{definition}\label{charvec-def}
Let $e=\{y_i,y_j\}$ be an edge of
$G$. The {\it characteristic vector\/} of $e$ is the vector
$v=e_i+e_j$, where $e_i$ is the $i^\textup{th}$ unit vector in $\mathbb{R}^n$. 
\end{definition}

In what follows $\mathcal{A}=\{v_1,\ldots,v_s\}$ will denote the 
set of all characteristic vectors of the edges of the graph $G$. We
may identify the edges of $G$ with the variables $t_1,\ldots,t_s$ of a
polynomial ring $K[t_1,\ldots,t_s]$ and
refer to $t_1,\ldots,t_s$ as the edges of $G$. 

\begin{definition} If $X$ is 
the projective algebraic
toric set parameterized by $y^{v_1},\ldots,y^{v_s}$, we call $X$ the
{\it projective algebraic toric set parameterized by the edges} of
$G$.
\end{definition}

\begin{definition} A graph $G$ is called {\it bipartite\/} if its vertex 
set can be partitioned into two disjoint 
subsets ${V}_1$ and ${V}_2$ such 
that every edge of $G$ has one end in ${V}_1$ and one end in ${V}_2$. The pair 
$(V_1,V_2)$ is called a {\it bipartition\/} of $G$.  
\end{definition}

Let $G$ be a graph. A vertex $v$ (resp. an edge $e$) of $G$  is called a 
{\it cutvertex\/} (resp. {\it
bridge\/}) if the number of connected components
of $G\setminus\{v\}$ (resp. $G\setminus\{e\}$) is larger than that of
$G$. A 
maximal connected subgraph of $G$ without cutvertices 
is called a {\it block\/}. A graph $G$ 
is $2$-{\it connected\/} if $|V_G|>2$ and $G$ has no 
cutvertices. Thus a block of $G$ is either a maximal 
$2$-connected subgraph, a bridge or an isolated vertex. By their maximality, 
different blocks of $G$ intersect in at most one vertex, 
which is then a cutvertex of $G$. Therefore every 
edge of $G$ lies in a unique block, and $G$ is the 
union of its blocks  (see \cite[Chapter~III]{Boll} for details).  

We come to the main result of this section.

\begin{theorem}\label{maria-jorge-vila-12-12-12} 
Let $G$ be a bipartite graph without isolated vertices
and let $G_1,\ldots,G_c$ be the blocks of $G$. If $K$ is a finite
field with $q$ elements and $X$ $($resp $X_k)$
is the projective algebraic toric set parameterized by
the edges of $G$ $($resp. $G_k\mathrm{)}$, then
$$
{\rm reg}\, K[E_G]/I(X)=\sum_{k=1}^c{\rm reg}\,
K[E_{G_k}]/I(X_k)+(q-2)(c-1).
$$
\end{theorem}

\begin{proof} We denote the set of all 
characteristic vectors of the edges of $G$ by
$\mathcal{A}=\{v_1,\ldots,v_s\}$. Let $P$ be the toric ideal of
$K[\{y^{v}\vert v\in\mathcal{A}\}]$, let $\mathcal{A}_k$ be the set of
characteristic vectors of the edges of $G_k$ and let $P_k$ be the toric ideal 
of $K[\{y^{v}\vert v\in\mathcal{A}_k\}]$. The toric ideal $P$ is the
kernel of the epimorphism of $K$-algebras
\[ 
\varphi\colon S=K[t_1,\ldots,t_s] \longrightarrow K[\{y^{v}\vert v\in\mathcal{A}\}], \ \ \ 
t_i\longmapsto y^{v_i}.
 \]
Permitting an abuse of notation, we may denote the edges of $G$ by
$t_1,\ldots,t_s$. As $G$ is a bipartite graph and
$E_{G_i}\cap E_{G_j}=\emptyset$ for $i\neq j$, from 
\cite[Proposition 3.1]{Vi3}, it follows that $P=P_1+\cdots+P_c$. 
Setting 
$$\mathcal{I}=(\{t_i^{q-1}-t_j^{q-1}\vert\, t_i,t_j\in E_G\})
\ \mbox{ and }\ \mathcal{I}_k=(\{t_i^{q-1}-t_j^{q-1}\vert\,
t_i,t_j\in E_{G_k}\}),
$$
by \cite[Corollary 2.11]{algcodes}, we get
$$
\textstyle((P+\mathcal{I})\colon \prod_{t_i\in E_G}t_i)=I(X)\ \mbox{
and }\ \textstyle((P_k+\mathcal{I}_k)\colon \prod_{t_i\in
E_{G_k}}t_i)=I(X_k).$$
Therefore the formula for the regularity follows from
Theorem~\ref{vila-12-12-12}.
\end{proof}

This result is interesting because it reduces the computation of the
regularity to the case of $2$-connected bipartite graphs. 
Next, we compare the ideals $I(X)$ and $I=P+\mathcal{I}$, where
$\mathcal{I}$ is the ideal $(\{t_i^{q-1}-t_j^{q-1}\vert\, 
t_i,t_j\in E_{G}\})$, and relate the
regularity of $I(X)$ with the Hilbert function of $S/I$ and the
primary decompositions of $I$. 

\begin{proposition}\label{dec30-12} Let $G$ be a bipartite graph which is not a
forest and let $P$ be the toric ideal of $K[y^{v_1},\ldots,y^{v_s}]$.
If $I=P+\mathcal{I}$ and $X$ is the projective algebraic toric set parameterized
by the edges of $G$, then the following hold\/{\rm:}
\begin{itemize}
\item[\rm(a)]  $I\subsetneq I(X)$ and $I$ is not unmixed.

\item[\rm(b)] There is a minimal primary decomposition 
$I=\mathfrak{p}_1\cap\cdots\cap\mathfrak{p}_m\cap\mathfrak{q}'$, where
$\mathfrak{p}_1,\ldots,\mathfrak{p}_m$ are prime ideals such 
that $I(X)=\mathfrak{p}_1\cap\cdots\cap\mathfrak{p}_m$ and
$\mathfrak{q}'$ is an $\mathfrak{m}$-primary ideal. Here
$\mathfrak{m}$ denotes the maximal ideal $(t_1,\ldots,t_s)$.  
\item[\rm(c)] If $t^a\in\mathfrak{q}'$, then $\mathfrak{q}:=I+(t^a)$
is $\mathfrak{m}$-primary, $(I\colon t^a)=I(X)$ and
$I=I(X)\cap\mathfrak{q}$.   
\item[\rm(d)] If $i_0$ is the least integer $i\geq |a|$ such
that $H_I(i)-H_\mathfrak{q}(i)=|X|$, then ${\rm reg}\,
S/I(X)=i_0-|a|$.
 \end{itemize}
\end{proposition}

\begin{proof} (a): Since $G$ has at least one even cycle of length at
least $4$, using \cite[Theorem 5.9.]{evencycles} it follows that
$I\subsetneq I(X)$. To show that $I$ is not unmixed, we proceed by
contradiction. Assume that $I$ is unmixed, i.e., all associated
primes of $I$ have height $s-1$. Then, by Lemma~\ref{apr24-12-1}, 
$I$ is a lattice ideal, i.e., $I$ is equal to $(I\colon(t_1\cdots t_s)^\infty)$, a
contradiction because $G$ is bipartite and according to
\cite[Corollary 2.11]{algcodes} one 
has $(I\colon(t_1\cdots t_s)^\infty)=I(X)$.

(b): As $I$ is graded, by (a), there is a minimal primary
decomposition
$I=\mathfrak{q}_1\cap\cdots\cap\mathfrak{q}_m\cap\mathfrak{q}'$, where
$\mathfrak{q}_i$ is $\mathfrak{p}_i$-primary of height $s-1$ for all
$i$ and $\mathfrak{q}'$ is $ \mathfrak{m}$-primary. By
Lemma~\ref{apr24-12-1}, $\mathfrak{q}_i$ is a lattice ideal for all
$i$. Hence
$$
I(X)=(I\colon(t_1\cdots
t_s)^\infty)=\mathfrak{q}_1\cap\cdots\cap\mathfrak{q}_m.
$$
As $I(X)$ is a radical ideal, so is $\mathfrak{q}_i$ for
$i=1,\ldots,m$, i.e., $\mathfrak{q_i}=\mathfrak{p}_i$ for all $i$.

(c): Let
$I=\mathfrak{p}_1\cap\cdots\cap\mathfrak{p}_m\cap\mathfrak{q}'$ be a
minimal primary decomposition as in $(b)$. Pick any monomial $t^a$ in 
$\mathfrak{q}'$. Then, by Lemma~\ref{apr24-12-1},
$\mathfrak{q}=I+(t^a)$ is $\mathfrak{m}$-primary and 
$$
(I\colon
t^a)=(\mathfrak{p}_1\colon t^a)\cap\cdots\cap (\mathfrak{p}_m\colon
t^a)\cap (\mathfrak{q}'\colon t^a)
=\mathfrak{p}_1\cap\cdots\cap\mathfrak{p}_m=I(X).
$$
From the equality $(I\colon t^a)=I(X)$, it follows 
readily that $I=I(X)\cap\mathfrak{q}$.

(d): Let $t^a$ be any monomial of $\mathfrak{q}'$ and let
$\ell=\deg(t^a)$. If $\mathfrak{q}=I+(t^a)$, by (c), there is an
exact sequence
$$
0\longrightarrow S/I(X)[-\ell]\stackrel{t^a}{\longrightarrow}
S/I\longrightarrow S/\mathfrak{q}\longrightarrow 0.
$$
Hence, by the additivity of Hilbert functions,
$H_X(i-\ell)=H_I(i)-H_\mathfrak{q}(i)$ for $i\geq 0$. Since $I(X)$ is
Cohen-Macaulay of dimension $1$, ${\rm reg}\, S/I(X)$ is equal to the
index of regularity of $S/I(X)$. Thus, ${\rm reg}(S/I(X))$, is the
least integer $r\geq 0$ such that $H_X(d)=|X|$ for $d\geq r$. Thus,
$r=i_0-|a|$.   
\end{proof}

\begin{theorem}{\rm(\cite{evencycles},
\cite{rs-codes})}\label{upper-lower-bounds-reg-bip}  
Let $G$ be a connected bipartite graph with
bipartition $(V_1,V_2)$ and let $X$ be the projective algebraic toric
set parameterized by 
the edges of $G$. If $|V_2|\leq|V_1|$, then 
$$
(|V_1|-1)(q-2)\leq{\rm reg}\, S/I(X)\leq (|V_1|+|V_2|-2)(q-2).
$$
Furthermore, equality on the left occurs if $G$ is a complete bipartite
graph or if $G$ is a Hamiltonian graph and equality on the right
occurs if $G$ is a tree. 
\end{theorem}

For an arbitrary bipartite graph, Theorems~\ref{maria-jorge-vila-12-12-12} and
\ref{upper-lower-bounds-reg-bip} can be used to bound the regularity
of $I(X)$. 

\medskip

{\bf Acknowledgments.} 
The authors would like to thank the referee for its 
careful reading of the paper and for showing us the right formulation
of Lemma~\ref{jul8-12} and a short proof of this lemma. 

\bibliographystyle{plain}

\end{document}